\documentclass[11pt]{amsart}
\usepackage {amsmath, amssymb,   epsfig, a4wide, enumerate, psfrag}
\usepackage[utf8]{inputenc}
\usepackage[english, frenchb]{babel}

\date{\today}

\author{Pierre Berger}
\author{Romain Dujardin}

\title{On stability and hyperbolicity for polynomial automorphisms of $\cd$}

\address{LAMA  \\
Universit\'e Paris-Est Marne-la-Vall\'ee  \\
5 boulevard Descartes \\
77454 Champs sur Marne  \\
France}
\email{romain.dujardin@u-pem.fr}

\address{LAGA \\
 Institut Galilée \\ 
 Université Paris 13\\
 99 avenue J.B. Clément \\ 93430 Villetaneuse \\
 France}
 \email{berger@math.univ-paris13.fr}
 
\newcommand{\cc}{\mathbb{C}}

\newcommand{\dd}{\mathbb{D}}

\newcommand{\nn}{\mathbb{N}}
\newcommand{\pp}{\mathbb{P}}
\newcommand{\e}{\varepsilon}
\newcommand{\cv}{\rightarrow}

\newcommand{\fr}{\partial}
\newcommand{\om}{\Omega}
\newcommand{\set}[1]{\left\{#1\right\}}
\newcommand{\norm}[1]{\left\Vert#1\right\Vert}
\newcommand{\abs}[1]{\left\vert#1\right\vert}
\newcommand{\cd}{{\cc^2}}

\newcommand{\rest}[1]{ \arrowvert_{#1}}

\newcommand{\unsur}[1]{\frac{1}{#1}}

\newcommand{\lrpar}[1]{\left(#1\right)}

\newcommand{\la}{\lambda}
\newcommand{\lo}{{\lambda_0}}
\newcommand{\loc}{{\mathrm{loc}}}

\newcommand{\La}{\Lambda}

\newcommand{\itm}{\item[-]}

\DeclareMathOperator{\modul}{mod}
\DeclareMathOperator{\jac}{Jac}
\DeclareMathOperator{\tub}{Tub}

\newtheorem{prop}{Proposition} [section]
\newtheorem{thm}[prop] {Theorem}
\newtheorem{defi}[prop] {Definition}
\newtheorem{lem}[prop] {Lemma}
\newtheorem{cor}[prop]{Corollary}
\newtheorem{theo}{Theorem}

\newtheorem{claim}[prop]{Claim}

\newtheorem{coro}[theo]{Corollary}
\newtheorem{defprop}[prop]{Definition-Proposition}

\theoremstyle{remark}
\newtheorem{exam}[prop]{Example}
\newtheorem{rmk}[prop]{Remark}

\begin{document}
\selectlanguage{english}

\begin{abstract} 
Let $(f_\la)_{\la\in \La}$ be a holomorphic family of polynomial automorphisms of $\cd$. Following previous work
of Dujardin and Lyubich,
we say that such a family is weakly stable if saddle periodic orbits do not bifurcate. It is an open question whether
this property is equivalent to structural stability on the Julia set $J^*$ 
 (that is, the closure of the set of saddle periodic points). 
 
 In this paper we introduce a notion of regular point for a polynomial automorphism, inspired by Pesin  theory, 
  and prove that in   a weakly stable family, the set of  
   regular points moves holomorphically. It follows that a weakly stable family
  is probabilistically  structurally stable, in a very strong sense.  
 Another consequence of these techniques is that  weak stability preserves uniform hyperbolicity on $J^*$. 
\end{abstract}

\maketitle

\section{Introduction}

Let $(f_\la)_{\la\in \La}$ be a holomorphic family of polynomial automorphisms of $\cd$, with non-trivial dynamics\footnote{A necessary and sufficient condition for this is that the {\em dynamical degree} $d = \lim (\deg(f^n))^{1/n}$  satisfies $d\geq 2$,
 see \S\ref{sec:prel} for more details.},   parameterized by a connected complex manifold $\La$. 
A basic  stability/bifurcation dichotomy in this setting was introduced by M. Lyubich and the second author  in \cite{dl}.
 In that paper it was proved in particular that under a moderate dissipativity 
  assumption\footnote{that is, the complex Jacobian $\mathrm{Jac}(f)$ satisfies $\abs{\mathrm{Jac}(f)}<d^{-2}$.}, 
stable parameters together with parameters exhibiting a homoclinic tangency form a dense subset of $\La$.
This confirms in this setting  a (weak version of a) well-known conjecture of Palis.   
The notion of stability into consideration here is the following: a family is said to be {\em weakly stable} if  
 periodic orbits do not bifurcate. Specifically, this means that  the eigenvalues of the differential do not cross the unit circle. 

In one-dimensional holomorphic dynamics, this  seemingly weak notion of stability actually leads to the usual  one of
 structural stability (on the Julia set or on the whole sphere) thanks to the theory of {\em holomorphic motions} 
developed independently by Mañé, Sad and Sullivan and Lyubich \cite{mss, lyubich-bif, lyubich-bif2}. 

As it is well-known, the basic theory of holomorphic motions breaks down in dimension 2, and 
 a corresponding notion of {\em branched holomorphic motion}  (where collisions are allowed), 
was designed in \cite{dl}. 
To be more specific, let  $J^*$ be the closure of the set of saddle periodic orbits. It was shown by  Bedford, Lyubich and Smillie 
that $J^*$ contains all homoclinic and heteroclinic intersections of saddle points, and conversely, if $p$ is any saddle point, then
$W^s(p)\cap W^u(p)$ is dense in $J^*$. 
It was proved in \cite{dl} that if $(f_\la)_{\la\in \La}$ is weakly stable,  then  there is an  
 equivariant 
  branched holomorphic motion of $J^*$, that  is unbranched over the set of periodic points 
  and homoclinic (resp. heteroclinic) intersections. This 
  means that such points have a unique holomorphic continuation in the family, and furthermore, this continuation cannot collide with other 
  points in $J^*$ (see below \S \ref{subs:prel autom} for more details).  The underlying idea  is that the motion is unbranched on sets satisfying 
  a local (uniform) expansivity property.

  Still, it remains an open question whether a weakly stable family is structurally stable on $J^*$. 
A weaker version of this question, which is natural in view of  the above analysis, 
 is whether the unbranching property   holds  
  generically with respect to hyperbolic invariant probability measures.
 
 \medskip
 
The first main goal in this paper is to answer this second question.  We introduce a notion  
of {\em regular point},    simply defined as follows: $p\in J^*$ is 
regular if there exists a sequence of saddle points $(p_n)_{n\geq 1}$ converging to $p$ such that $W^u_\loc(p_n)$ 
and $W^s_\loc(p_n)$ are of  size uniformly bounded from below   as $n\cv\infty$
and do not asymptotically coincide
(see below \S \ref{sec:regular} for the formal definition, and \S\ref{subs:param1} for the notion of the local size of a manifold). 
The set $\mathcal{R}$ of regular points is invariant and dense in $J^*$ since it contains saddle points and homoclinic intersections. 
More interestingly,   Katok's closing lemma \cite{katok} implies that $\mathcal R$ is of full mass relative to 
any hyperbolic invariant probability measure. Observe however that our definition of regular point makes 
no reference to any invariant measure. Notice also that in our context, thanks to the Ruelle inequality,
any invariant measure with positive entropy is hyperbolic.

\medskip

Our first main result is the following.

\begin{theo}\label{theo:pesin strong}
Let $(f_\la)_{\la\in \La}$ be a  substantial family of polynomial automorphisms of $\cd$ of dynamical degree $d\geq 2$, 
that is weakly stable.   

Then the set of regular points moves holomorphically and without collisions. More precisely, for every $\la\in \La$, every regular point of $f_\la$ 
admits a unique continuation under the branched motion of $J^*_\la$, which remains regular in the whole family. In particular, 
the restrictions  $f_\la\rest{\mathcal R_\la}$ are topologically conjugate. 
\end{theo}

The meaning of the word ``substantial" will be explained in \S\ref{subs:prel autom} below; 
it will be enough for the moment to note that any dissipative family is substantial by definition.
 By ``topologically conjugate"  we mean that  there exists a homeomorphism $h: \mathcal R_\la \cv \mathcal R_{\la'}$
 such that $h\circ f_\la = f_{\la'}\circ h $ in restriction to $\mathcal R_\la$. 

  Let us say that a     polynomial automorphism $f$ is {\em probabilistically structurally stable} (in some 
given family $(f_\la)$) if for every $f'$ sufficiently close to $f$, there exists a 
set $\mathcal{R}_f$ (resp. $\mathcal{R}_{f'}$) which is of full measure with respect to 
any hyperbolic invariant probability measure for $f$ (resp. $f'$) together with a continuous 
conjugacy  $\mathcal{R}_f\cv\mathcal{R}_{f'}$. 

Recall also from the work of  Friedland and Milnor \cite{fm} that every dynamically non-trivial 
polynomial automorphism is conjugate to a composition of Hénon mappings. 

Together with \cite[Thm A \& Cor. 4.5]{dl}, Theorem \ref{theo:pesin strong} enables us to 
go one step further in the direction of the   Palis Conjecture mentioned above. 

\begin{coro} 
Let $f$ be a composition of Hénon mappings in $\cd$. Then:
\begin{itemize}
\item $f$ can be approximated in the space of polynomial automorphisms of degree $d$ 
either by a    probabilistically
structurally stable map, or by one  possessing infinitely many sinks or sources.
\item   If $f$ is moderately dissipative and not probabilistically structurally stable, then
$f$ is a limit of automorphisms displaying  homoclinic
tangencies.
\end{itemize}
\end{coro}
  
The main step of the  proof of Theorem \ref{theo:pesin strong} consists in studying how  
the size of local stable and unstable manifolds of a given saddle point varies in a weakly stable family. More precisely, 
assume that for some $\lo\in \La$, $p(\lo)$ is a saddle point such that $W^s(p(\lo))$ has bounded geometry at scale $r_0$ at $p(\lo)$. Since  
$(f_\la)$ is weakly stable, $p(\lo)$ persists as a saddle point $p(\la)$ for ${\la \in \La}$. In \S\ref{sec:size}, we give estimates on the geometry 
of $W^s_{\rm loc}(p(\la))$ {\em which depend only on $r_0$}, based on the extension properties 
of the branched holomorphic motion of $J^*$ along  unstable manifolds devised in \cite{dl}.  

These estimates are used to control the geometry of the local
  ``center  stable manifold" of $\set{(\la,p(\la)), \ \la\in \La}$, which is of codimension 1 in $\La\times \cd$. 
  With this codimension 1 subset at hand, we can 
  prevent collisions between the motion of points in $J^*$  using  classical tools from complex geometry
   (like the persistence of proper intersections and the Hurwitz Theorem).

 \medskip

We actually prove a more general version of Theorem \ref{theo:pesin strong}, which involves only regularity in one of 
 the stable or the unstable directions (see Theorem \ref{thm:regular} below). One motivation for this is that  in the dissipative setting 
 it is possible in certain situations to take advantage of 
 dissipativity to obtain a good control on the geometry of stable manifolds (see Example \ref{exam:zero}).

\medskip

If $f$ is uniformly hyperbolic on $J^*$, then it is well known that $f\rest{J^*}$ is structurally stable. In particular, if
$(f_\la)_{\la\in \La}$ is any family of polynomial automorphisms, and $\lo\in \La$ 
is such that $f_\lo$ is uniformly hyperbolic on $J^*_\lo$, then  $(f_\la)$ is (weakly) stable
 in some neighborhood of $\lo$. 
Thus, $\lo$ belongs to   a {\em hyperbolic component} in $\La$, where this uniform hyperbolicity is preserved, which is itself contained in  a possibly larger {\em weak stability component}. Our next main result asserts that these two components actually coincide.

\begin{theo}\label{theo:hyp}
Let $(f_\la)_{\la\in \La}$ be a  substantial family of polynomial automorphisms of $\cd$of dynamical degree $d\geq 2$, 
that is weakly stable. 
Assume that there exists $\lo\in \La$ such that $f_\lo$ is uniformly hyperbolic on $J^*_\lo$. Then for every $\la\in \La$, 
$f_\la$ is uniformly hyperbolic on $J^*_\la$.
\end{theo}

As a result, this theorem enables to identify the phenomena responsible for the breakdown of uniform hyperbolicity
in a family of polynomial automorphisms:  hyperbolicity  can only be destroyed  
by the bifurcation of some saddle orbit to a  sink  or a source   
(which by \cite{dl} implies the creation of  homoclinic tangencies, in the moderately dissipative setting). 

The proof of Theorem \ref{theo:hyp} relies on the techniques of Theorem \ref{theo:pesin strong}, together with a geometric criterion for hyperbolicity due to Bedford and Smillie \cite{bs8}.

\medskip

The plan of the paper is the following. In \S\ref{sec:prel} we  discuss the notion of weak stability, following \cite{dl}. We also establish some 
preliminary results on sequences of subvarieties in $\cc^d$. In \S\ref{sec:size} we study how 
the geometry of unstable manifolds  varies in a weakly 
stable family. In \S\ref{sec:regular} we introduce the notion of regular point and prove Theorem \ref{theo:pesin strong}, 
and finally \S\ref{sec:hyperbolic} is devoted to the proof of Theorem \ref{theo:hyp}. 

Throughout the paper, we make the standing assumption that the parameter space $\La$ is the unit disk. In view of Theorems \ref{theo:pesin strong} and \ref{theo:hyp} this is not a restriction since we can always connect any two parameters in $\La$ by a chain of holomorphic disks. 
We also 
use the classical convention $C(a, b, \ldots)$ to denote a constant which depends only on the previously defined quantities $a$, $b$, etc.
 
 \smallskip
 
 {\bf Acknowledgments.} This research was partially supported by the ANR project LAMBDA,  ANR-13-BS01-0002 and the
Balzan project of J. Palis.
 
\section{Preliminaries}\label{sec:prel}

In this section we collect  some basic facts  on polynomial automorphisms of $\cd$, and give a brief account on the notion of weak stability
introduced in  \cite{dl}. We also establish some preliminary results on sequences of analytic subsets.  

\subsection{Families of polynomial automorphisms of $\cd$} \label{subs:prel autom}
Let us start with some standard facts about the iteration of an individual polynomial automorphism $f$ of $\cd$ 
(see \cite{bs1, bls} for more details and references). The {\em dynamical degree} is an integer $d$ defined by $d= \lim_{n\cv\infty} (\deg(f^n))^{1/n}$, and  $f$ has non-trivial dynamics if and only if $d\geq 2$. It is then conjugate to a composition of generalized Hénon mappings $(x,y)\mapsto (p(x)+ay,ax)$. Here are some dynamically defined subsets:
\begin{itemize}
\item $K^\pm$ is the set of points with bounded forward orbits under $f^{\pm1}$.
\item $K= K^+\cap K^-$ is the filled Julia set.
\item $J^\pm = \fr K^\pm$ are the forward and backward Julia sets. Stable (resp. unstable) manifolds of saddle periodic points are dense in $J^+$ (resp. $J^-$).
\item $J^*\subset J = J^+\cap J^-$ is the closure of the set of saddle periodic points. Saddle points and   homoclinic and heteroclinic intersections are contained (and dense) in $J^*$. 
\end{itemize}
The Green functions $G^\pm$ are defined by $G^\pm(z) = \lim d^{-n} \log^+\norm{f^n(z)}$, and  are non-negative  continuous   plurisubharmonic functions. They  are pluriharmonic whenever positive and
$K^\pm$ coincides with  $\set{G^\pm = 0} $. 
 The associated currents are $T^\pm = dd^cG^\pm$ whose supports are  $J^\pm$. 
 If $\Delta\subset \cd$ is a  holomorphic disk, then $G^+\rest{\Delta}$ is harmonic iff $T^+\wedge [\Delta] = 0$ iff
  $\lrpar{f^n\rest{\Delta}}_{n\geq 1}$ is a normal family (equivalently $\Delta\subset K^+$ or $\Delta\subset \cd\setminus K^+$). 
 
 \medskip
 
 Let now $(f_\la)_{\la\in \La}$ be a holomorphic family of polynomial automorphisms with fixed dynamical degree $d\geq 2$, parameterized 
 by a connected complex manifold.  We will use the notation $K_\la$, $J^*_\la$, etc. to denote the corresponding dynamical objects. If a preferred parameter $\lo$ is given  
 we often simply use the subscript `0' instead of $\lo$.   
 
 To  the  family $(f_\la)$ is associated a fibered dynamical system in $\La\times \cd$ defined by $\widehat f:(\la, z)\mapsto (\la, f_\la(z))$. Then we mark with a hat the corresponding fibered objets, e.g. $\widehat K = \bigcup_{\la\in \La} \set{\la}\times K_\la$, etc. 
 
 Such a family is always conjugate to a family of compositions of Hénon mappings \cite[Prop. 2.1]{dl}.
It follows that the  sets $K_\la$ are locally uniformly bounded in $\cd$.
 
 \medskip
 
From now on we report on some results from \cite{dl}.
 A family of polynomial automorphisms of dynamical degree $d\geq 2$ is said {\em substantial} it:
 either all its members are dissipative or for any periodic point with eigenvalues $\alpha_1$ and $\alpha_2$, 
 no relation of the form $\alpha_1^a\alpha^b_2 = c$ holds persistently in parameter space, 
 where $a$, $b$ , $c$ are complex numbers and $\abs{c} =1$. From now on, we assume without further notice that all families have constant dynamical degree $d\geq 2$ and are substantial. 
 
 A {\em branched holomorphic motion} $\mathcal G$ is a family of holomorphic graphs over $\La$ in $\La\times \cd$. 
 All branched holomorphic motions 
 considered in this paper are locally uniformly bounded, so in particular they form normal families 
 (we then say that $\mathcal G$ is {\em normal}).  A branched holomorphic motion $
 \mathcal G$ is {\em unbranched} along $\gamma$ if $\gamma$ does not cross any other graph in the family. 
 If it is unbranched along any graph $\gamma$, then it is by definition a {\em holomorphic motion}. 
 If $\mathcal G$ is normal,  closed and unbranched at $\gamma$, and if $(\gamma_n)_{n\geq 0}
 \in \mathcal G^\nn$ is any sequence such that for 
 some $\lo\in \La$, $\gamma_n(\lo)\cv \gamma (\lo)$, then $\gamma_n\cv\gamma$. We thus see that unbranching along $\gamma$ is a 
 form of continuity of the motion. We can make this precise as follows: if $\mathcal G$ is a (non-necessarily closed) normal holomorphic motion and  $\overline{\mathcal{G}}$ is unbranched at $\gamma_0$, then $\mathcal{G}$ is continuous at $\gamma_0$. 
 
 \medskip
 
 A substantial family $(f_\la)_{\la\in \La}$ of polynomial automorphisms is said to be  {\em weakly stable} if every periodic point  stays of 
 constant type (attracting, saddle, indifferent, repelling) in the family. Equivalently, $(f_\la)$ is weakly stable if the sets $J^*_\la$ move under 
 an equivariant  branched holomorphic motion. A central theme in this paper will be to show that this motion is unbranched at certain points. 
 In this respect, the following result is essential.
 
 \begin{thm}[see Cor. 4.12 and Prop. 4.14 in \cite{dl}]\label{thm:unbranched}
 Let $(f_\la)_{\la\in \La}$ be a weakly stable substantial family of polynomial automorphisms of dynamical degree $d\geq 2$. If for 
 $\lo\in \La$, $p(\lo)$ is a saddle point or a homoclinic or a heteroclinic intersection, then it admits a 
 unique continuation $p(\la)$ which remains of the same type, and the branched holomorphic motion of $J^*$ is unbranched along $p$. 
\end{thm}

The motion of $J^*$ can be extended to a branched holomorphic motion of $J^+\cup J^-$, using the density of  
  stable and unstable manifolds of saddles.  The details are as follows (for concreteness we deal with 
   unstable manifolds, of course analogous results hold in the stable direction). 
  The global unstable manifold of a saddle point is parameterized by $\cc$. 
More precisely, in our situation, for every $\la$ there exists an injective holomorphic  immersion $\psi^u_\la:\cc\cv\cd$ such that 
$\psi^u_\la(0) = p_\la$ and for $\zeta \in \cc$, $f_\la\circ \psi^u_\la (\zeta)= \psi^u(u_\la \zeta)$, where $u_\la$ denotes the unstable multiplier. Such a $\psi^u$ is unique up to pre-composition with a linear map, and will be referred to as an {\em unstable parameterization}. In addition, the normalization of 
$\psi^u_\la$ may be chosen so that  $(\la,\zeta)\mapsto \psi^u_\la(\zeta)$ is holomorphic. 
The precise way to do it is irrelevant for the moment; we shall have to  discuss this issue more carefully later on. 

Thanks  to these parameterizations, we can use the theory of holomorphic motions in $\cc$ to derive information about the motion of unstable manifolds in $\cd$. 
The following is a combination of Proposition 5.2 and Lemma 5.10 in \cite{dl}.

\begin{prop}
Under the above hypotheses there exists a natural  equivariant holomorphic motion  
 $h_\la: W^u(p_0) \cv W^u(p_\la)$, with $h_0 = \mathrm{id}$, 
 that   respects the decomposition    $$W^u(p)=(W^u(p)\cap U^+) \sqcup (W^u(p)\cap K^+).$$
 \end{prop}
 
 Beware that we are not claiming that  the points in $W^u(p)$  have a unique continuation, only that they have a {\em natural} one. 
This motion is constructed by taking the canonical  extension of the motion of homoclinic intersections 
(this is due to Bers and Royden \cite{bers royden}).   

Notice that the notation 
$h_\la$ here refers to the  motion of points in $\cd$. Given a holomorphic family of parameterizations $\psi_\la^u$ of
$W^u(p_\la)$, it will also be of interest in some situations  to consider the corresponding holomorphic motion 
$(\psi^u_\la)^{-1}\circ h_\la$ in $\cc$, which we will   denote by $h_\la^u$. 

\medskip

These holomorphic motions need not preserve the levels of the function $G^+$: indeed this is already the 
case\footnote{For instance  the value of the Green function at
 critical points is in general not invariant in a $J$-stable family 
of polynomials.} for 
$J$-stable families in dimension 1. 
The following  easy lemma asserts that   $G^+$ admits  locally uniform distortion  
 along the motion. It thus   provides a    link between the intrinsic (i.e. inside unstable manifolds) and the extrinsic 
 properties of the motion, and   will play an important role in the paper. 

\begin{lem}\label{lem:harnack}
For every compact subset $\widetilde \La\Subset \La$ 
there exists a constant $C = C(\widetilde\La) \geq 1$  such that for every 
 $z \in W^u(p_0)$
$$ \unsur{C} G^+_0(z) \leq  G^+_\la(h_\la(z))\leq  {C}  G^+_0(z). $$
\end{lem}

 \begin{proof} 
 Recall that for any holomorphic family of polynomial automorphisms of degree $d$, the function 
 $(\la,z)\mapsto G^+_\la(z)$ is plurisubharmonic in $\La\times \cd$, jointly continuous in $(\la,z)$, and pluriharmonic where it is positive (see \cite[\S3]{bs1}).

 If $z\in K^+_0$, then $h_\la(z)\in K^+_\la$ and $G^+_\la(h_\la(z))\equiv 0$ 
 so there is nothing to prove. If $z\notin K^+_0$ then 
 $\la\mapsto G^+_\la(h_\la(z))$ is a positive harmonic function, so the result follows from the Harnack inequality 
 \cite[Thm 3.1.7]{hormander}.
 \end{proof}

\subsection{Sequences of analytic subsets}\label{subs:prel geom}
 To make the paper accessible to readers potentially not  so familiar with complex geometry, let us first  recall a few 
  classical facts on complex analytic sets, and sequences  of such objects.
The reader is referred to the book of Chirka \cite{chirka} for more details. 

\medskip

Let $\om\subset \cc^d$ be a connected open set.
A (complex) \emph{analytic subset} or \emph{subvariety}  $A$ of $\om$ is a subset of $\om$ that is 
 covered by open sets $U$ of $\mathbb C^n$, for which there exist
   $p\ge 0$ and a holomorphic map $\phi \colon  U\mapsto \mathbb C^p$, such that
$A\cap U= \{z\in U,\; \phi(z)=0\}$.
  
A point $a\in A$ is \emph{regular} if there exists a neighborhood $U$ of $a$ so that $A\cap U$ is a (complex) submanifold. The set of regular points of $A$ is denoted by $\mathrm{Reg}(A)$, and its complement 
 $\mathrm{Sing}(A)= A\setminus\mathrm{Reg} (A)$ is  the  \emph{singular set}. 
 A subvariety is \emph{smooth} if its singular set is empty. 

 An \emph{irreducible component}  of $A$ is the closure of a connected component of ${\mathrm Reg}(A)$. 
It is itself an analytic set. The \emph{dimension} of an analytic subset
 $A$ is the maximal dimension of its irreducible components. It is said of {\em pure dimension} if all its irreducible components
 have the same dimension.
 
 A {\em hypersurface} (resp. a {\em curve}) is an analytic subset of pure codimension (resp. dimension) 1, possibly singular.  

\medskip

We recall that the Hausdorff distance $d_{HD}$ between two closed subsets $E$ and $F$ of  a metric space is infimum of $m\in [0,\infty]$ 
such that $E$ is included in the $m$-neighborhood $F$ and {\it vice-versa}. Let $\Omega$ be an 
open subset of $\mathbb C^d$.  A sequence of closed subsets $(A_j)_j$ of $\Omega$ converges to a closed subset $A\subset \Omega$, if 
for every compact set $K$ of $\Omega$, it holds $d_{HD}(K\cap A_j, A\cap K)\to 0$. The set of closed (resp. closed and connected) 
subsets of $\Omega$ endowed with 
Hausdorff distance is relatively compact.  

A key ingredient to study the convergence of analytic subsets sequences is the following classical result known as Bishop's Theorem:

\begin{thm}[see \cite{chirka} p. 203]\label{thm:bishop}
Let $(A_j)_j$ be a sequence of pure $p$-dimensional subvarieties  of 
 an open subset $\Omega\subset \mathbb C^d$,  converging to  a (closed) subset $A \subset \Omega$ and such that the $2p$-dimensional Hausdorff measure (that is, the $2p$-dimensional volume) 
$m_{2p}(A_j)$  is locally uniformly bounded:
\[\forall K\Subset \Omega,\;  \exists M_K>0,\;    \forall j, \; m_{2p}(A_j\cap K)<M .\]
Then $A$ is also a pure $p$-dimensional subvariety of $\Omega$.
\end{thm}

In particular the set of subvarieties with locally uniformly bounded volume is compact. 
We can actually be more precise about the convergence  in Bishop's Theorem. 
Let $A_n$ be a sequence of analytic sets with uniformly bounded volumes converging in the 
Hausdorff topology to an irreducible analytic set $A$. Then there exists a positive integer $m$, the 
 {\em multiplicity of convergence}, which can be described as follows.
 If $p\in \mathrm{Reg}(A)$ is any regular point of $A$, and $N$ is a compact 
neighborhood of $p$ in which $(A\cap N, N)$ is biholomorphic to $(\dd^k\times \set{0}, \dd^d)$, then for $n$ large enough, $A_n
\cap N$ is a branched cover over  $A\cap N$ of degree $m$. In particular if $m=1$, $A_n
\cap N$ is a graph over  $A\cap N$

\medskip

Let us  now state a few results which will be used many times in the paper. The following 
 result can be interpreted as a kind of abstract version of the $\La$-lemma of \cite{mss}. 

\begin{prop}\label{prop:cv} 
Let $\om\subset \cc^d$ be a connected open set. 
Let $(V_n)$ be a sequence of analytic subsets of codimension 1 in $\om$ with uniformly bounded volumes. 

Assume that:
\begin{itemize}
\itm the $V_n$ are disjoint;
\itm there exists $p_n\in V_n$ such that $p_n\cv p\in \om$;
\itm every cluster value of $(V_n)$ is locally irreducible at $p$. 
\end{itemize}
Then the sequence $(V_n)$ converges. 
\end{prop}

The irreducibility assumption is necessary in this result, as shown by the sequence of curves in $\cd$ 
defined by $V_{2n} = \set{x=0}$ and $V_{2n+1} = \set{xy = 1/n}$. 

\begin{proof} Assume that $V = \lim V_{n_j}$ and $W = \lim V_{n'_k}$ are distinct cluster limits of $(V_n)$. Then $V$ and $W$ are irreducible and contain $p$, therefore they must intersect non-trivially at $p$. Since $V$ and $W$ are of codimension 1, they intersect properly, that is 
$\dim(V\cap W)=d-2$. Now, proper intersections are robust under perturbations (see prop. 2 p. 141 and  cor. 4 p. 145 in 
\cite{chirka}), so we infer that $V_{n_j}$ and $ V_{n'_k}$ intersect non-trivially  for large $j$ and $k$, which is contradictory. 
\end{proof}

\medskip

In general the limit of a sequence of smooth hypersurfaces can be singular. The smoothness of the limit can be ensured in 
certain circumstances (compare \cite[Prop. 11]{lyubich peters}). 

\begin{prop}\label{prop:disk}
Let $(V_n)$ be a sequence of curves with uniformly bounded area in the unit ball of $\cd$, which converges   to $V$. Assume that for every $n$, $V_n$ is biholomorphic to a disk. Then $V$ is irreducible. If in addition the multiplicity of convergence is 1, then $V$ is 
 smooth.
\end{prop}

\begin{proof}
Fix $p\in V$. Let us first show that $V$ is locally irreducible at $p$. 
  Fix a small ball $B$ about $p$ and let  $V \cap B= V ^1\cup \cdots \cup V ^q$ be 
 the decomposition into (local) irreducible components. Shrinking  $B$ slightly if necessary, we may
 assume that each $V^i$ contains $p$ and that $\fr B\cap \overline{V^i}$ is not empty.
 Likewise, we may assume that $V$ is smooth near $\fr B$ and transverse to it. Hence 
 $V\cap \fr B$ is a union of disjoint smooth (real) curves $(C_j)_j$ and each $C_j$ is contained in a unique irreducible component $V^i$.

For every $n$, let $V_n^i$ be a connected component of $V_n\cap B$.  By the uniform bound on the area and Bishop Theorem, we can 
extract a  subsequence $(V_{n'}^i)_{n'}$ converging  to a curve 
 $A$. Since $A$ is included in $V\cap B$, it is an union of irreducible components of $V\cap B$.

Observe  that each $V_n^i$ is a holomorphic disk, as follows from 
 the maximum principle applied to the subharmonic function $z\mapsto \norm{\varphi_n(z)- p}$, where $\varphi_n$ is a parametrization of $V_n$. The boundary $C_n^i= \overline{V_n^i}\cap\fr B$ is homeomorphic to a 
  circle for every $n$,  hence $(C_{n'}^i)_{n'}$ converges to a connected compact set.
  Our assumptions on $\fr B\cap V$   
   imply that $C_{n'}^i$ is close to a unique component
   $C^i$ of $V\cap \fr B$. On the other hand the loop $C^i$ is in the boundary of a unique irreducible 
   component of $V\cap B$, say $V^i$. Thus $A=V^i$ and $\overline A\cap \fr B= \overline{V^i}\cap \fr B=C_j$. 

For every $n'$, consider  any connected component $V_{n'}^j$ of $V_{n'}\cap B$, so that $V_{n'}^j$ converges to a certain irreducible component $V^j$ of $V$. If $V^i$ and $V^j$ are not  equal, they intersect  properly, hence 
 the same occurs for $V_{n'}^i$ and $V_{n'}^j$, a contradiction. 
 This proves that $V$ is locally irreducible and $\overline{V}\cap \fr B$ is a single loop $C_j$. 

 %
		%
   
  \medskip
  
  Now assume that the multiplicity of convergence is 1 and let us show that $V$ is smooth. Assume by contradiction that $V$ is singular at $p$. By the multiplicity 1 convergence hypothesis and the transversality of $V$ and $\fr B$, 
  for large $n$ the loop $C_{n'}^i=\overline{V_{n'}^i}\cap\fr B$ is smooth, close to  $C^i=\overline{V}\cap\fr B$ and  (smoothly) isotopic to   it.
	
	The (smooth) {\em genus} of $C^i$ is by definition the smallest genus of a smooth surface in $B$ bounded by $C^i$. It is   
	invariant under smooth isotopy. It is known that if $V$ is singular at $p$, then for a sufficiently 
	small ball $B= B(p, r)$ around $p$, the genus of $V\cap \fr B$ is positive 
   (see \cite[Cor 10.2]{milnor}). Since $V_{n'}^i$ is a holomorphic disk, 
   we arrive at a contradiction, which finishes the proof. 
\end{proof}

\section{Uniform geometry of (un)stable manifolds}\label{sec:size}

In this  section we consider a weakly stable substantial family $(f_\la)_{\la\in \La}$ of polynomial automorphisms of 
dynamical degree $d\geq 2$. 
Let $\lo\in \La$, and fix a saddle periodic  point $p_0$ for $f_0$. 
By weak stability, $p_\la$ persists  as a saddle point throughout the family. Our purpose is to give uniform  estimates on the 
geometry of $W^{s/u}_{\rm loc}(p_\la)$, depending only of that of $W^{s/u}_{\rm loc}(p_0)$.
For concreteness, from now on we deal with unstable manifolds. 
Recall that  $\La$ was assumed to be  the unit disk. 

We present two types of results, which will both be used afterwards. In \S\ref{subs:area} we show that the area of the local 
unstable manifold of $p_\la$ can be controlled throughout $\La$. The techniques here are reminiscent of the results of 
\cite[\S 3]{bs8}. We introduce a notion of size of a manifold at a point in \S\ref{subs:param1}, and show that unstable parameterizations can 
be controlled in term of the size of $W^u(p)$ at $p$. Finally in 
  \S\ref{subs:param2} we show that the size    of $W^u (p_\la)$ at $p_\la$ is uniformly bounded from below in a 
neighborhood of $\lo$ that  depends only on the size of $W^{s/u} (p_0)$ at $p_0$. 

\subsection{Areas of local unstable manifolds} \label{subs:area}
By definition we say that $D\subset \cd$ is a {\em  holomorphic disk} if there is a holomorphic map  $\phi:\dd\cv \cd$  with  $\phi(\dd) = D$,   which extends to a homeomorphism  $\overline \dd \cv \overline D$.

 In the next lemma we give a basic estimate on  the geometry of a holomorphic disk in $\cd$,  relying     
  on simple ideas from conformal geometry.
The modulus of 
an annulus, will be  denoted by    $\modul(A)$. 

\begin{lem}\label{lem:geometry}
Let $D_1 \Subset D_2$ be a pair of holomorphic disks in $\cd$ with $0\in D_1$ and 
let $d_1, d_2, m$ be positive real numbers such that 
$$ \sup_{  z\in D_2} \norm{z}\leq d_2,\ \sup_{ z\in   \fr D_1} \norm{z}\geq d_1,
 \text{ and } \modul(D_2\setminus \overline{D_1})\geq m.$$ 

Then there exist positive  constants $A$ and $r$ depending only on $d_1$, $d_2$ and $m$ such that the connected component of $D_2\cap B(0,r)$ containing 0 is a properly embedded submanifold in $B(0, r)$, of area not greater than $A$. 
\end{lem}

Notice that  if $D$ is a holomorphic disk, 
    $\fr D$ refers to the   boundary of  $D$ relative to its intrinsic topology. 
Notice also that the maximum principle applied to the subharmonic function $z\mapsto \norm{z}^2$ on $D$ implies that 
$ \sup_{  z\in D} \norm{z} = \sup_{  z\in \fr D} \norm{z}$.

\begin{proof}
Fix a biholomorphism $\phi : \dd\cv D_2$ with  $\phi(0) = 0$. We claim that there exists $\delta>0$ depending only on $m$ such that 
$\phi^{-1} (D_1)\subset D(0, 1-\delta)$. Indeed by assumption $\modul(\dd\setminus \overline{\phi^{-1} (D_1)})\geq m$. Now it follows from a classical result of Grötzsch that if $U$ is a connected and simply connected open subset of $\dd$ containing 
$0$ and $z$ with $|z|=:1-x\in (0,1)$
 then $\modul(\dd\setminus \overline U)\leq \modul(\dd\setminus [0,1-x])$ (see Ahlfors \cite[Thm 4-6]{ahlfors}). In addition, 
the map $x\mapsto \rho(x):= \modul(\dd\setminus [0,1-x])$ is increasing and continuous. 
Taking the contraposite, we see that if $\modul\lrpar{\dd\setminus \overline{\phi^{-1}(D_1)}}\geq m$, 
then  $\phi^{-1} (D_1)$ is contained in $ D(0, 1-\delta)$ with  $\delta:=\rho^{-1}(m)$, and we conclude that 
$\phi$ satisfies 
 $\sup_{  \dd} \norm{\phi}\leq d_2$ and  $\sup_{  \fr D(0, 1-\delta)} \norm{\phi}\geq d_1$.
 The result then follows from Lemma \ref{lem:compactness} below.
\end{proof}

\begin{lem}\label{lem:compactness}
Let $\phi: \dd\cv\cd$ be a holomorphic mapping fixing $0$ and such that $\sup_{\dd}\norm{\phi}\leq d_2$ and $\sup_{\fr D(0, 1-\delta)}\norm{\phi}\geq d_1$. Then there exist constants $A$ and $r$ depending only on $d_1$ and $d_2$ 
such that  the connected component of $\phi(D(0, 1-\delta))\cap B(0, r)$ 
containing 0 is a properly embedded submanifold in $B(0, r)$, of area not greater than $A$.
\end{lem}

\begin{proof}
This is an elementary compactness argument.  Indeed  let us 
   show that for every such $\phi$ there exists a uniform $r$ such that   the connected component $C$ of $\phi^{-1}\lrpar{ {B(0,r)}}$ containing 0 is relatively compact  in $D(0, 1-\delta)$, that is $\overline  C\subset D(0, 1-\delta)$.
	The area bound in turns  follows from the Cauchy inequality.
To prove that such a $r$ exists, for the sake of contradiction we suppose the existence of a sequence $(\phi_n$ of such functions which 
		violate this property for $r_n\to 0$. Hence there exists for every $n$ a connected compact set $C_n\subset  \overline D(0, 1-\delta)$ of 
diameter $\ge 1-\delta$  sent into $\overline B(0,r_n)$ by $\phi_n$. We can suppose that  $(C_n)_n$ converges to a connected compact set 
$C_\infty$ of diameter $\ge 1-\delta$ and that $(\phi_n)_n$ converges uniformly to a certain $\phi_\infty$ on $\overline D(0, 1-\delta)$. Then $
\phi_\infty$ vanishes on $C_\infty$ hence on $\dd$. This  contradicts the fact that
		$\sup_{\fr D(0, 1-\delta)}\norm{\phi_\infty}\geq d_1$. 
\end{proof}


Lemmas \ref{lem:geometry} and \ref{lem:compactness} 
can be combined to estimate how the geometry of an unstable manifold varies 
in  a weakly stable family. For a saddle point $p$ and a positive real number $r$, we denote by $W^u_r(p)$ 
the connected component of $W^u(p)\cap B(p,r)$ containing $p$, which by the maximum principle  is a holomorphic disk. 

\begin{prop}\label{prop:uniform}
Let $(f_\la)_{\la\in \La}$ be a weakly stable substantial family of polynomial automorphisms of $\cd$ of
dynamical degree $d\geq 2$. Fix $\lo\in \La$ and a saddle periodic  
point $p_0$ for $f_0$, and denote by $(p_\la)_{\la\in \La}$ its continuation. 

Consider a pair  $D_1\Subset D_2$ of  holomorphic disks  in $W^u(p_0)$, with $p_0\in D_1$, and let 
$$g_1 = \sup (G^+_0\rest{ D_1}), \ g_2 = \sup(G^+_0\rest{D_2}) \text{ and } m = \modul(D_2\setminus \overline{D_1}).$$
Then for every $\widetilde \La\Subset \La$, there exist positive constants $r$, $g$ and $A$ depending only on $\widetilde \La$, $g_1$, $g_2$ and $m$ such that for every  $\la\in \widetilde \La$, $W^u_r  (p_\la)$ is a properly embedded submanifold into 
$B(p_\la, r)$, contained in $h_\la(D_1)$,  whose  area is not greater than $A$, and such that  $\sup (G_\la^+\rest{W^u_r  (p_\la)})\geq g$. 
\end{prop}

Observe that $G^+$ does not vanish identically in any neighborhood of $p$ in $W^u_{\rm loc}(p)$ so $g_1, g_2$ are indeed positive.  
 
\begin{proof}
Fix $\widetilde\La\Subset \La$. 
For $\la\in \widetilde\La$,   consider the disks $h_\la(D_1)$ and $h_\la(D_2)$.
The quasiconformality of holomorphic motions implies that 
$$\mod\lrpar{h_\la(D_2) \setminus \overline {h_\la(D_1)}}\leq C m,$$ where $C$ 
depends only on $\widetilde\La$. In addition, it follows from Lemma \ref{lem:harnack} 
that $$\sup \left( G^+_\la\rest{h_\la(D_2)}\right)\leq C' g_2 \text{ and }
\sup \left(G^+_\la\rest{h_\la(\fr D_1)}\right) = \sup \left(G^+_\la\rest{h_\la(D_1)}\right)\geq (C')^{-1} g_1,$$ where again $C'$ 
depends only on $\widetilde\La$.

Now recall  that $G^-_\la(z)$ is jointly continuous in $(\la,z)$ and that  for every $\la$, $G^+\rest{K^-_\la}$ is proper. 
Since unstable manifolds are contained in $K^-$ we   infer that 
 there exists $d_2$ depending only on $g_2$, $\widetilde \Lambda$ and the family $(f_\lambda)_\lambda$
such that for $\la \in \widetilde \Lambda$:
$$\sup_{z\in h_\la(D_2)} \norm{z-p_\la}
\leq d_2\;.$$

Also it is known  that the Green function $G^+$ is  Hölder continuous (see \cite[Thm 1.2]{fs}). 
Moreover the proof of \cite{fs}  easily  shows  that the modulus of continuity of $G^+$ is locally uniform in $\La$. 
Therefore, $G^+_\la(p_\la)=0$ implies the existence of $d_1$ depending only on $g_1$, 
$\widetilde \Lambda$ and the family $(f_\lambda)$ 
such that for  $\la \in \widetilde \Lambda$:
$$\sup_{z\in h_\la(\fr D_1)} \norm{z-p_\la} \geq d_1.$$
	Applying Lemma \ref{lem:geometry}   finishes the proof. 
\end{proof}

\subsection{Estimates  on unstable parameterizations and applications}\label{subs:param1}
Endow $\cd$ with its natural Hermitian structure. A  {\em bidisk of size $r$} is the  image of $D(0, r)^2$ by some affine 
isometry. The image of the unit bidisk under a general affine map will be referred to as an {\em affine bidisk}. 
A curve $V\subset \cc$ is a graph over an affine line $L$ if its orthogonal projection onto $L$ is injective restricted to $V$.
Then we have a well-defined notion of slope of a holomorphic curve with respect to $L$.  

\begin{defi}
A curve $V$ through $p$ is said 
to have \emph{bounded geometry at scale $r$ at $p$}
 (we also simply say that \emph{$V$ has size $r$ at $p$})
if there exists a neighborhood of $p$ in $V$ that is a graph of slope at most 1   
over a disk of radius $r$ in the tangent space $T_pV$. 
\end{defi}

Let $V$ be a disk of size $r$ at $p$, and fix orthonormal coordinates $(x,y)$ so that $p=0$ and $T_pV = \set{y=0}$. Then 
 the connected component of $V$ through $p$ in the bidisk $ D(0,r)^2$ is a graph $\set{y=\varphi(x)}$
 over the first coordinate with $\abs{\varphi'}\leq 1$ and $\phi'(0)=0$. 
\begin{rmk}\label{pentereduite} 
The Schwarz lemma implies that for every $x\in D(0, r)$, 
 $\abs{\varphi'(x)}\leq \abs{x}/r$. 
 \end{rmk}
 

It will be a key fact for us that  the Koebe Distortion Theorem provides 
estimates on unstable  parameterizations   in terms of the size of local unstable 
 manifolds (see also Lemma \ref{lem:D12} below).

 \begin{lem}\label{lem:unstable}
 Let $f$ be a polynomial automorphism of $\cd$ and $p$ a saddle periodic point. Assume 
  that $W^u(p)$  is of size $r$ at $p$. Normalize the coordinates   so that $p=(0,0)$ and 
  $W^u(p)$ is tangent to the $x$-axis at $p$.   
  Denote by    $\pi$    the first coordinate projection  and let  
       $\Gamma^u(p)$ be the component of $\pi^{-1}(D(0,r))\cap W^u(p)$ containing $p$.

   Let $\psi^u:\cc\cv\cd$ be an unstable parameterization,   such  that  $\psi^u(0)  = p$, and $\norm{(\psi^u)'(0)}=1$.
Then $\psi^u\lrpar{D\lrpar{0, {\frac{r}{4}}}}\subset\Gamma^u(p)\subset  D(0,r)^2$. 
Moreover for every 
$\abs{z}\leq\frac{r}{8}$, 
\begin{equation}\label{eq:distortion}
   D\lrpar{0, { \frac{\abs{z}}{4} }}\subset \pi\circ \psi^u\lrpar{D\lrpar{0,\abs{z} }}  \subset    D\lrpar{0, 4\abs{z} }. 
\end{equation}
\end{lem}

 \begin{proof} Without loss of generality, rotate the first coordinate so that $(\pi\circ \psi^u)'(0) = 1$. 
 Under the assumptions of the lemma, $\pi\circ\psi^u$ is   univalent   from some unknown domain $\om\subset \cc$ onto $D(0,r)$.  
Now recall the Koebe Distortion Theorem (see \cite[Thm 5-3]{ahlfors}): 
    if $g:\dd\cv\cc$ is a univalent mapping, with $g'(0)=1$, 
    then  for $z\in \dd$, 
\begin{equation}\label{eq:distortion2}
\frac{\abs{z}}{4} \leq \frac{\abs{z}}{(1+\abs{z})^2} \le 
\abs{g(z)} \leq \frac{\abs{z}}{(1-\abs{z})^2}.  
\end{equation}
Applying this to $g(z) = r^{-1} (\pi\circ\psi^u)^{-1}(rz)$, we first deduce that  
$(\pi\circ\psi^u)^{-1} (D(0,r))\supset D\lrpar{0, {\frac{r}{4}}}$, 
 thus 
$\psi^u\lrpar{D\lrpar{0, {\frac{r}{4}}}}\subset D(0,r)\times \mathbb C$ and so 
$\psi^u\lrpar{D\lrpar{0, {\frac{r}{4}}}}\subset\Gamma^u(p)$.
It follows that  the function $h$ in $\dd$ defined by 
$ \zeta\mapsto h(\zeta)  =  \frac{4}{r} \pi\circ\psi^u\lrpar{\frac{r\zeta}{4}}$ is univalent and satisfies $h'(0) =1$. 
Applying \eqref{eq:distortion2} 
to $h$ yields \eqref{eq:distortion}, as desired. 
  \end{proof}

 Another important idea in this paper is that of the {\em natural continuation} of an unstable parameterization. 
 Let us explain what this is about. 
 Fix a parameter $\lo\in \La$,  a saddle point $p_0$ for $f_0$, and an unstable parameterization $\psi^u_0:\cc\cv\cd$ (in practice  
we often choose it so that $\norm{(\psi^u_0)'(0)} =1$). 
 We want to find a well-adapted
holomorphic family of parameterizations $\psi^u_\la$ of $W^u(p_\la)$, with $\psi^u_\la(0)  = p_\la$. 
 Since the Bers-Royden extension is canonical,  the motion in $\cd$ of a given point 
 $q_0\in W^u(p_0)$ (denoted by $q_\la$) does not depend on this choice of parameterizations.
  Fix such a point $q_0$, say   $ q_0  = \psi^u_0(1)$. 	   We now  fix the parameterization of  $W^u(p_\la)$ by declaring that 
  $(\psi^u_\la)^{-1}(q_\la) = (\psi^u_0)^{-1}(q_0) = 1$, or equivalently
 $(\psi^u_\la)^{-1}(q_\la)= 1$. 
This is by definition the {\em natural continuation} of $\psi^u_0$ in the family.

 Such a holomorphic family of parameterizations can be constructed
 from any given holomorphic family $\widetilde\psi^u_\la$ by the formula
 $$\psi^u_\la(z ) = \widetilde\psi^u_\la \lrpar{{(\widetilde\psi^u_\la)^{-1}(q_\la)} z}.$$

 The advantage is now that the holomorphic motion $h_\la^u$ in $\cc$ defined by looking at the motion 
  of points 
  in the coordinate $\psi^u_\la$, that is, 
 $h_\la^u(z) = (\psi^u_\la)^{-1}(h_\la(\psi^u_0(z)))$, is normalized by $h_\la^u(0) = 0$ and $h_\la^u(1)  = 1$. 
It is well known that such a normalized holomorphic motion in $\cc$ satisfies uniform bounds : 
  for every $\widetilde \La\Subset \La$, there exists constants $A$, $B$ and $\alpha$ depending only on $\widetilde \La$ such that if 
  $\la,\la'\in \La$, 
\begin{equation}\label{eq:holmotion}
\abs{h_\la^u(z) - h_{\la'}^u(z')}\leq A \rho(z,z')^\alpha+ B\abs{\la-\la'},
\end{equation}
 where $\rho$ denotes the spherical metric (see \cite[Cor. 2]{bers royden}). 

\medskip

We now use these techniques to give an estimate on parameterizations which supplements  Proposition \ref{prop:uniform}.

\begin{prop}\label{prop:uniformparam}
Let $(f_\la)_{\la\in\La}$ be a weakly stable  substantial  family   of polynomial automorphisms of $\cd$ and let
 $(p_\la)_{\la\in\La}$ be a holomorphically moving saddle point. 
 Assume that  for $\la=\lo$,  $W^u(p_\lo)$ is of size $r_0$ at $p_0$. 
 Let $\psi^u_\lo$ be an unstable parameterization of $W^u(p_\lo)$ and $(\psi^u_\la)_{\la\in \La}$ be its natural continuation. 
 
 Then for every $\widetilde \La\subset \La$ there exist  constants $c$ and    $M$ depending only on $\widetilde \La$ and $r_0$ such that 
 if $\la\in \widetilde \La$, 
 $ \norm{\psi^u_\la}\leq M$ 
 on $D(0, c r_0)$. 
 \end{prop}
 
 \begin{proof}
  By the H\"older continuity property of $G^+$, there exists $g = g(r_0)>0$ depending only on $r_0$  such that
		$   \sup \lrpar{G^+_\lo\rest{W^u_{r_0\sqrt2}(p_\lo)}}\leq g$ (recall that a bidisk of radius $r_0$
		 is contained in a ball of radius $r_0\sqrt2$). 
Hence by Lemma \ref{lem:unstable} we deduce that $G^+_\lo\rest{\psi^u_\lo(		D(0, r_0/4))}\leq g$, thus 
  Lemma \ref{lem:harnack} implies that $G^+_\la \rest{\psi^u_\la(h_\la^u(D(0, r_0/4)))}\leq C g$ where 
   $C$ depends only on $\widetilde \La$.    By the properness of $G^+\rest{K^-}$, we   deduce that  
   $ \psi^u_\la(h_\la^u(D(0, r_0/4)))$ is uniformly bounded by $M(\widetilde \La, r_0)$ in $\cd$. 
   Finally, since $h_\la^u(0) = 0$,  by \eqref{eq:holmotion} we infer that if $\la\in \widetilde \La$, then
   $h^u_\la(D(0, r_0/4))$ contains  $D(0, cr_0)$ for some 
 $c = c(r_0, \widetilde \La)$, and we conclude that    $\norm{ \psi^u_\la(h_\la^u(D(0,cr_0)))}\leq M$ for $\la\in \widetilde \La$, which was the 
 desired result. 
 \end{proof}

\subsection{Local persistence of the size of unstable manifolds}\label{subs:param2}
 Recall the notation  $\widehat{p} = 
 \set{(\la, p_\la), \ \la\in \La}$ for a holomorphically moving saddle point $p_\la$. 
 Also, let us  denote by $\tub(\widehat p, r)$ the fibered tubular  neighborhood of
  $\widehat p$ of radius $r$ in 
 $\La\times \cd$, defined by 
 $$ \tub(\widehat p, r)  = \set{(\la, z)\in \La\times \cd, \ \norm{z-p(\la)}<r}.$$

Let us first isolate a geometric lemma. For notational ease, we put $\cc^2_\la = \set{\la}\times \cd$. 

\begin{lem}\label{lem:smooth}
Let $\widehat p$ be the graph of a holomorphic mapping $p:\La\cv\cd$ that is uniformly bounded by $M$ 
on $\La$. Fix a direction $e\in \pp^1(\cc)$.
Fix domains  $\widetilde U$ and $U$ such that $\widetilde U \Subset U \Subset \La$
and 
assume that $\Sigma$ is a hypersurface that is closed in $\tub(\widehat p, r)\cap (U\times \cd)$ and such that 
for each $\la\in U$, $\Sigma \cap  \cc^2_\la$ 
is a graph of slope at most 1 over $p(\la)+ e$. Then $\Sigma$ is smooth and the volume of 
$\Sigma\cap \lrpar{\widetilde U\times \cd}$ is bounded 
by a constant depending only on $M$, $\widetilde U$, $U$ and $r$.
\end{lem}

\begin{proof} Identify $e$ and the corresponding line through 0 in $\cd$. 
 For $\la \in U$, let 
  $\phi_\lambda$ be the unique holomorphic map
   from an open subset of $p(\la)+e$ to its orthogonal complement,  whose graph is $\Sigma \cap  \cc^2_\la$. 
 Let $\pi\colon  \mathbb C ^2 \to \mathbb C$ be the orthogonal projection on the line $e$, and let $\hat \pi(\lambda, z) = (\lambda,\pi(z))$. 
 
By continuity of the map $(\la,z)\mapsto \phi_\la(z)$, the following is an open subset of $U\times e\approx U\times \cc$.
\[\hat \pi(\Sigma)=  \{(\la,z)\in U\times e:\; |\phi_\la (z)|^2+|z|^2<r^2\}.\]
By the graph property, the map $\hat \pi\rest{\Sigma}$ is one-to-one from the subvariety $\Sigma$ onto the open subset $\hat \pi(\Sigma)$. 
Under these conditions it is  classical      that $\hat \pi\rest{\Sigma}$ 
is a biholomorphism (see Prop. 3 p. 32 in \cite{chirka}). In particular $\Sigma$ is smooth and the map $ (\la ,z)\in \hat \pi(\Sigma')\mapsto 
\phi_\la(z)$ is holomorphic in both variables (being the inverse of $\hat \pi$).

To get the volume bound, 
we remark that the set $\hat \pi(\Sigma)$ is bounded (specifically, it is  
contained in $U\times B(0,M+r)$). The derivative  $\partial_ z\phi_\la$ is bounded by 
$1$, and the image of $\phi_\la$ is bounded by $M+r$. The Cauchy estimate  implies that the derivative  $\partial_\la\phi_\la$ is bounded on $
\widetilde U$. Consequently the volume of  $\Sigma\cap \big(\widetilde U\times \cd\big)$ 
is bounded by a constant depending only on $M$, $\widetilde U$, $U$ and $r$. 
\end{proof}

The main result in this subsection is that in a weakly stable family, the size of a holomorphically moving unstable manifold is locally uniformly bounded from below.

 \begin{prop}\label{prop:surface}
Let $(f_\la)_{\la\in\La}$ be a weakly stable   substantial family   of polynomial automorphisms of $\cd$ and let
 $\widehat p = (p_\la)_{\la\in\La}$ be a holomorphically moving saddle point. Let $\widetilde\La\Subset \La$ be a relatively compact open subset and fix $\lo\in \widetilde\La$.
 Assume that  for $\la=\lo$,  $W^u(p_0)$ is of size $r_2$ at $p_0$.
 Then, for every $r_1<r_2$, there exists  $\delta  = \delta\big( r_1,r_2, \widetilde\La \big)$ depending only on $r_1$, $r_2$ 
 and $\widetilde{\La}$ such that if $\abs{\la-\lo}<\delta$, 
 $W^u(p_\la)$ is of size $r_1$ at $p_\la$, and $ W^u_{r_1}  (p_\la)$ is a graph of slope at most 1 over 
$p_\la+ E^u(p_0)$ (where $E^u(p_0)$ denotes the unstable direction at $p_0$).

 Furthermore, there exists a  submanifold $\widehat W^u_{r_1}$ in  $\tub(\widehat p, r_1) \cap (D(\lo,\delta)\times \cd)$,    
such that for every $\la\in D(\lo,\delta)$, 
$\widehat W^u_{r_1}\cap \mathbb{C}^2_\la = W^u_{r_1} 
 (p_\la)$, whose volume is bounded by a constant $V\big(r_1, r_2, \widetilde \La\big)$ 
 depending only on $r_1$,  $r_2$ and $\widetilde\La$.  
   \end{prop}

\begin{proof} 
Start with an  unstable parameterization  $\psi^u_0$ of $W^u(p_\lo)$ satisfying $\norm{(\psi^u_0)'(0)} =1$, and 
let $\pi_0$ be the orthogonal projection onto $E^u(p_0)$. For $i=1,2$, we  
 denote by  $D_i = (\pi_0\circ \psi_0^u)^{-1}\lrpar{ D\lrpar{0, r_i}}$,  and let 
 $D_{12} := (\pi\circ \psi_0^u)^{-1}\lrpar{ D\lrpar{0, r_{12}}}$, with $r_{12}:= (r_1+r_2)/2$.
 These are simply connected domains in   $\cc$ containing the origin, satisfying $D_1\Subset D_{12}\Subset D_2$. 
 
 \medskip
 
The following lemma will be proved afterwards.

\begin{lem}\label{lem:D12}
For every $z\in D(0, r_{12})$, 
the following derivative estimate holds: 
\begin{equation}\label{eq:D12}
\frac{1-r_1/r_2}{16}\leq  \abs{  \lrpar{ \lrpar{\pi_0\circ \psi^u_0}^{-1}}'(z)}\leq \frac{16}{(1-r_1/r_2)^3}.
\end{equation}
Moreover the  distance between  $D_1$ and $\fr D_2$ is greater than $  r_2 (1-r_1/r_2)^2/32$.
\end{lem}

Let now $(\psi^u_\la)_{\la\in\La}$ be the natural continuation of $\psi^u_0$. The second assertion of 
  Lemma \ref{lem:D12} together with \eqref{eq:holmotion}   imply that 
 there exists $\delta=\delta\big( r_1,r_2, \widetilde\La \big)$ such that if $\abs{\la-\lo}<\delta$,  $h_\la^u(D_1)$ 
stays uniformly far from $\fr (h_\la(D_2))$ (farther than $  r_2 (1-r_1/r_2)^2/50$, say) relative to the Euclidean metric on $\cc$.	
 Furthermore, arguing exactly as in Proposition \ref{prop:uniformparam}, we see 
 that for $\abs{\la- \lo}<\delta$, $\psi^u_\la(h_\la^u(D_2)))$ is uniformly bounded in $\cd$.
    
      Let $\Psi:\La\times \cc \cv\La\times \cd$ be defined 
  by $\Psi(\la, z) = (\la, \psi^u_\la(z))$, and put $$\widehat D_i = \bigcup_{\la\in B(\lo, \delta)}\set{\la}\times h_\la^u(D_i), \text{ for } i=1,2.$$
   Since $h_\la^u( D_1)$ stays far from $\fr( h_\la^u(D_2))$, and  
    $\Psi\big(\widehat D_2\big)$ is uniformly bounded in 
   $B(0, \delta)\times \cd$,    by the Cauchy estimates, reducing $\delta$ 
   again  slightly if necessary    the derivatives of   $\Psi$ are uniformly bounded on $\widehat D_1$, with  bounds depending only on $\widetilde \La$, $ r_1$ and $r_2$. 
   
    \medskip

We are now ready to conclude the proof. Let $\pi_\la$ (resp. $\pi_\la^\bot$)
 be the orthogonal projection onto $E^u(p_\la)$ (resp.  $(E^u(p_\la))^\bot$).
The curve $W^u(p_\la)$ is of size $r_1$ at $p_\la$ if for every $z$ 
in $D_1^\la :=\{z\in \cc,\; |\pi_\la \circ \psi^u_\la(z)|<r_1\}$ the following estimate holds:  
\begin{equation}\label{pente} |\partial_z(\pi_\la^\bot\circ \psi^u_\la)(z)|\le |\partial_z(\pi_\la\circ \psi^u_\la)(z)|.\end{equation}

By the Cauchy estimate on $\partial_\la \partial_z \Psi$,  $(\psi_\la^u)'(0)$ is close to  $(\psi^u_0)'(0)$ for $|\la-\lo|\le \delta$. In particular 
choosing $\delta = \delta(r_1, r_2, \widetilde\La)$ small enough we can ensure that $\norm{\pi_0 - \pi_\la}\leq \e$ (resp. 
$\norm{\pi_0^\bot - \pi_\la^\bot}\leq \e$), where  $\e$ is as small as we wish.  
By the Cauchy estimate on $\partial_\la \Psi$, for $\delta=\delta(r_1,r_2, \widetilde \La)$ sufficiently small, when   $|\la-\lo|\le \delta$, the set $D_1^\la$ is included in $D_{12} = (\pi\circ \psi_0^u)^{-1}\lrpar{ D\lrpar{0, r_{12})}}$, with $r_{12}:= (r_1+r_2)/2$. 
By Remark \ref{pentereduite}, for every $z\in  D_{12}$, we have that
\[ |{\partial_z(\pi_0^\bot\circ \psi^u_0)(z)}|\le \frac{r_1+r_2}{2 r_2} \abs{\partial_z(\pi_0\circ \psi^u_0)(z)}
 = \lrpar{1-\frac{r_2-r_1}{2r_2}}  \abs{\partial_z(\pi_0\circ \psi^u_0)(z)}.\]
 From this we infer that with $\e$ as above  and $z\in  D_{12}$,
  \begin{align}
   |{\partial_z(\pi_\la^\bot\circ \psi^u_\la)(z)}|\le \lrpar{1-\frac{r_2-r_1}{2r_2}}  & \abs{\partial_z(\pi_\la\circ \psi^u_\la)(z)} \label{eq:relou}  \\
  &+ \e \norm{\fr_z \psi_0^u(z)} + \e \norm{\fr_z \psi_\la^u(z)} + 2 \norm{  \fr_z (\psi_0^u - \psi_\la^u) (z)} \notag.
  \end{align}
 In addition, the right hand inequality in \eqref{eq:D12} implies that for $z\in D_{12}$, 
 $$\abs{\partial_z(\pi_0\circ \psi^u_0)(z)} \geq \frac{(1-r_1/r_2)^3}{16}.$$
By the Cauchy estimate on $\partial_\la \partial_z \Psi$, for $\delta= \delta(r_1,r_2, \widetilde \La)$ sufficiently small, 
a similar estimate holds for $\partial_z(\pi_\la\circ \psi^u_\la)(z)$ (with 16 replaced by 32, say) for 
 $|\la-\lo|\le \delta$ and $z\in D_{12}$. Recall that under our assumptions  $D_{12}$ contains  $D_1^\la$. Thus,  
  by choosing 
$\e = \e(r_1, r_2, \widetilde \La)$ appropriately and reducing $\delta$ again if necessary, we can ensure that 
for $z\in D^\la_1$, 
$$ \e \norm{\fr_z \psi_0^u(z)} + \e \norm{\fr_z \psi_\la^u(z)} + 2 \norm{  \fr_z (\psi_0^u - \psi_\la^u) (z)} \leq 
\frac{r_2-r_1}{2r_2} \abs{\partial_z(\pi_\la\circ \psi^u_\la)(z)},$$ which by \eqref{eq:relou} yields \eqref{pente}.
 
Finally, we define  $\widehat W^u_{r_1}$ to be the connected component of $\Psi\big(\widehat D_1\big)$ in   $\tub(\widehat p, r_1) \cap (D(\lo,\delta)\times \cd)$ containing $\widehat p$, which is a surface with the desired properties (its smoothness follows from Lemma \ref{lem:smooth}).
  \end{proof}
  
\begin{proof}[Proof of Lemma \ref{lem:D12}]
By the Koebe Distortion Theorem   (see \cite[Thm 5-3]{ahlfors}), if $g:\dd\cv\cc$ is a univalent mapping with $g'(0) =1$, then for every 
for every $r<1$ and every $z\in D(0,(1+r)/2)$ we have that 
$$  \frac{1-r}{16} \leq  \frac{1-|z|}{(1+|z|)^3} \leq |g'(z)| \leq  \frac{1+|z|}{(1-|z|)^3}\leq \frac{16}{(1-r)^3}$$
Applying this to   $g(z)=r_2^{-1}(\pi\circ \psi_0^u)^{-1}( r_2 z)$ and 
$r =r_1/r_2$, we deduce the desired bound on $((\pi\circ\psi^u_0){-1})'$. The estimate on the distance from 
 $D_1$ to $\fr D_2$    immediately follows. 
 \end{proof}


\begin{rmk}
One may wonder why we did not conclude to the existence of such a submanifold $\widehat W^u$ in $\widetilde\La\times \cd$ 
straight after 
Proposition \ref{prop:uniform}, using  a ``fibered" 
compactness argument in the style of  Lemma \ref{lem:compactness}.

The trouble is that in this general situation, having information about the area of $W^u_r(p_\la)$ 
is not sufficient to control the geometry (say, the volume) of $\widehat W^u$  because 
$\widehat W^u\cap( \set{\la}\times B(p_\la,r))$ can get disconnected for some values of $\la$, and the geometry of components 
other than $W^u_r(p_\la)$ can go out of control. Proposition \ref{prop:surface} shows that this phenomenon 
does not occur in some neighborhood of $\lo$, depending on the size of $W^u_{\rm loc}(p)$.  

\medskip

Let us briefly describe an explicit example where this phenomenon happens. 
Let $\phi:\dd\times \dd \times \cv \cd$ be defined by 
$\phi(\la, z)  = z(z-2\la) (z,g(z))$, where $g$ is a holomorphic function on $\dd$ such that
 $\abs{g}<1$ but $\int_\dd \abs{g'}^2 = \infty$.  
By  Lemma \ref{lem:geometry}, there exists positive constants $r$ and $A$
 such that for every $\la\in D(0, 3/4)$, the connected component $W_\la$
of $\phi(\la, \dd)\cap B(0,r)$ containing 0 is properly embedded and of area at most $A$. 

Now let  $\Phi: \dd\times \dd \cv \dd\times \dd^2$ be defined by  
$\Phi(\la, z)= (\la, \varphi(\la, z))$. We see that $\Phi(\dd\times\set{0}) = \dd\times \set{0}$. 
With  $r$ as above, consider the component $\widehat W$ of 
$\Phi(\dd^2)\cap \tub_r(\dd\times \set{0})$ containing $ \dd\times \set{0}$. 
Put  $V = \set{(\la,z), \; z(z-2\la) = 0}\subset \dd^2$, and observe that 
$\Phi(V) = \dd\times \set{0}$. 
Now it is easily shown that $\widehat{W}$ contains $\Phi(\tub_{r/4}(V))$. So when $\abs{\la}$ is close to 1/2, 
$\widehat{W} \cap  \dd^2_\la$ is made of at least two irreducible components, and for  the values of $\la$ such that 
$\int_{D(\la, r/4)} \abs{g'}^2 = \infty$, one of these is of  infinite volume. \qed
 \end{rmk}

\section{Holomorphic motion of regular points}\label{sec:regular}

In this section we introduce the concept of regular point for a polynomial automorphism of $\cd$, and 
prove Theorem \ref{theo:pesin strong}, in a slightly more general form. 

\subsection{Definitions and main statements}
\begin{defi}\label{defi:usregular}
We say that $p\in J^*$  is \emph{u-regular} (resp.  \emph{s-regular}) if there exists 
$r>0$ and a sequence of saddle periodic points $p_n$ converging to $p$,  with the property that $W^u(p_n)$ (resp. $W^s(p_n)$) 
 has bounded  geometry at scale  $r$  at $p_n$.
\end{defi}

If necessary, we make the size appearing in the definition explicit by speaking of ``u-regular point of size $r$".
The key property of u-regular (resp. s-regular)
 points is that they possess ``local unstable (resp. stable) manifolds", as the following proposition shows.

\begin{prop}\label{prop:cvsize}
  Let $f$ be a polynomial automorphism of $\cd$ with dynamical degree $d\geq 2$.
Let $p$ be a u-regular point of size $r$. 
  Then there exists a unique submanifold $W^u_r(p)$ of size $r$ at $p$ 
  such that  if $(p_n)$ is any   sequence of saddle 
  points   converging to  $p$, 
    such that   $W^u(p_n)$ is of size $r$ at $p_n$, the sequence of disks 
    $(W^u_r(p_n))$ converges to $W^u_r(p)$ with multiplicity 1 in $B(p,r)$. 
  In particular the unstable directions converge as well. 
  \end{prop}
  
  By definition  $W^u_r(p)$ will be referred to as the {\em local unstable manifold} of $p$ (and likewise for s-regular points). 
  If the size $r$ is not relevant (i.e. if we think of the local unstable manifold as a germ)
   we simply refer to it as $W^u_{\rm loc}(p)$.
  Let us stress that  we do \emph{not}
 claim that $W^u_{\rm loc}(p)$ is an unstable manifold in the usual sense. 
  
  \begin{proof}
 Fix $r'<r$. Then for $n\geq N(r')$, $W^u_r(p_n)\cap B(p,r')$ is a closed submanifold  in $B(p,r')$.
 Given any subsequence   $W^u_r(p_{n_j})$, up to further extraction  we may assume that 
 the $W^u_r(p_{n_j})$ are graphs of slope at most 2 over a fixed direction. It follows that 
  all cluster values of the sequence 
  $(W^u_r(p_n))$ are smooth, irreducible and of multiplicity 1. 
From Proposition \ref{prop:cv} we infer that  this sequence actually converges and the proof is complete.
   \end{proof}
  
  \begin{defi}\label{defi:regular}
We say that $p\in J^*$ is  \emph{regular} if it is both s- and u-regular and if its  local stable and unstable manifolds do not coincide at $p$. If  in addition these local stable and unstable manifolds are transverse, we say that $p$ is transverse regular. 
\end{defi}
  
Examples of transverse regular points include saddle periodic points, as well as  transverse homoclinic intersections (due to Smale's 
horseshoe construction). It follows from Katok's Closing Lemma that if $\nu$ is any hyperbolic ergodic
invariant probability measure (that is, whose Lyapunov exponents satisfy $\chi^-(\nu)<0<\chi^+(\nu)$), 
then $\nu$-a.e. point is transverse regular in the sense of Definition \ref{defi:regular}. 

Let us introduce a  weaker notion of regularity, which involves the stable direction only.

\begin{defprop}\label{defi:exposed}
Let $p\in J^*$ be a s-regular point. We stay that $p$ is \emph{s-exposed} if one of the following 
equivalent properties is satisfied:
\begin{itemize}
\item[$(i)$] $W^s_\loc(p)$ is not contained in $K$;
\item[$(ii)$] $G^-\rest{W^s_\loc(p)} \not\equiv 0$;
\item[$(iii)$] $T^-\wedge [W^s_\loc(p)] >0$;
\item[$(iv)$] for every saddle point $q$, the manifold $W^u(q)$ admits  transverse intersections with  $W^s_\loc(p)$.
\end{itemize}
 \end{defprop}
 
\begin{proof}
The equivalence between $(i)$,   $(ii)$ and $(iii)$ is clear. To see that $(iv)$ implies $(i)$, it suffices to notice that by the inclination lemma, a small neighborhood of $W^u(q)\cap W^s_\loc(p)$ in $W^s_\loc(p)$ cannot be included in $K$.
 The fact that $(iii)$ implies $(iv)$ follows from the techniques of
 \cite[\S 9]{bls}. The precise statement is that if $\Delta$ is any holomorphic disk 
 such that $T^-\wedge [\Delta] >0$, then $\Delta$ admits transverse intersection with 
 $W^u(q)$.  The case where $\Delta$ is contained in a stable manifold is  explained 
  in detail in \cite[Lemma 5.1]{dl}. The proof for a general holomorphic disk $\Delta$ is identical.
\end{proof}
 %

\begin{prop}
If $p\in J^*$ is regular, then it is s- and u-exposed. 
\end{prop}

\begin{proof}
It is enough to prove that $p$ is s-exposed. Let $(p_n)$ be a sequence of saddle points with $W^u(p_n)$ of size $r$ at $p_n$ converging to $p$. We assume that the sequence $(p_n)$ takes infinitely many values, the remaining case is easy and left to the reader. Then removing at most one term to this sequence we may assume that for every $n$, $p\notin W^u(p_n)$.  
We claim that for large $n$, $W^u_r(p_n)$ intersects transversally $W^s_{\rm loc}(p)$ at a point close to $p$.
 If $W^u_{\rm loc}(p)$ and  $W^s_{\rm loc}(p)$ are transverse this is clear. If not, since $W^u_r(p_n)\cap W^u_r(p)= \emptyset$, this follows from \cite[Lemma 6.4]{bls}. 
In any case, arguing
 as in the implication {\em (iv)}$\Rightarrow${\em (i)} of Proposition \ref{defi:exposed} 
we conclude that $W^s_\loc(p)$ is not contained in $K$ and we are done.  
\end{proof}

Here is a basic example:

\begin{exam}
If $p$ is a saddle point and $q$ belongs to  the boundary of  
$W^s(p)\cap K^-$  relative to the intrinsic topology of $W^s(p)$, then $q$ is 
 s-regular and exposed. 
Indeed it is shown in \cite[Lemma 5.1]{dl} that $q$ is the limit of a sequence of 
homoclinic intersections $(t_n)$, thus $q$ is exposed inside $W^s(p)$. 
Furthermore if $\Delta\subset W^s(p)$ is any  disk containing $p$ and $q$, 
 it follows from Smale's horseshoe construction that for every $n$, $t_n$ is a limit of a sequence of saddle points $(p_{n,k})_k$
 whose stable manifolds are graphs over $\Delta$. 
 By considering the diagonal sequence $p_{n,n}$ we conclude that $q$ is s-regular, as desired.\qed
\end{exam}

Also there are examples of points which are s-regular and exposed but a priori not regular:

\begin{exam}\label{exam:zero}
Let $f$ be a dissipative polynomial automorphism. 
Let $m$ be an ergodic probability measure supported on $J^*$ with the property that $m= \lim m_n$, where  for each $n$, $m_n$ is a 
probability measure equidistributed on a set of non-attracting periodic orbits (that is, saddle or semi-neutral). Since $f$ is dissipative, the 
negative Lyapunov exponent of $m_n$   satisfies $\chi^-(m_n) \leq \log\abs{\jac(f)} <0$, and likewise for $m$. 
On the other hand we make no assumption on the remaining (non-negative) Lyapunov exponent.

 Then it is possible\footnote{Details will appear in  subsequent work.} to adapt 
the techniques of Wang-Young \cite[\S 2]{WY} (see also Benedicks-Carleson \cite{BC91}) to show that if the Jacobian is sufficiently small, 
by the Pliss Lemma
there exists a set of periodic points $A_r$ such that  $m_n(A_r)\geq 1/2$ for each $n$,
 and such that for every $p\in A_r$ the local
 stable manifold of $p$ is of size $r$. 
 Thus the same holds for $m$, and by ergodicity we conclude that $m$-a.e. point is s-regular. Furthermore, since in this case the local stable manifolds obtained by Proposition \ref{prop:cvsize}  coincide with Pesin stable 
 manifolds, it follows that $m$-a.e. point is s-exposed. 
 
An interesting   example of such a situation is given by   the unique invariant probability measure supported on 
   the attractor of an infinitely renormalizable Hénon map (see \cite{clm}).  \qed
\end{exam}

Here is a more precise version of Theorem \ref{theo:pesin strong}.  

\begin{thm}\label{thm:regular}
Let $(f_\la)_{\la\in \La}$ be a weakly stable substantial family of polynomial automorphisms of $\cd$ of dynamical degree $d\geq 2$. 
If for some parameter $\lo$, $p_0\in J^*_\lo$ is  s-regular and exposed
for $f_\lo$, then there exists a unique holomorphic 
mapping  $\la\mapsto p(\la)$ such that  for every $\la$, $p(\la)\in K_\la$ and $p(\lo)=p_0$. Moreover,  for every $\la\in \La$, $p(\la)$ is 
s-regular and exposed. In particular the branched holomorphic motion of $J^*$ is unbranched along the curve $(\la, p(\la))$.
 \end{thm}

Using the terminology introduced in \cite[\S 3]{dl}, we can reformulate this by saying that the set $\mathcal{R}^s$ 
of s-regular and exposed  points moves under  a strongly unbranched, hence continuous, holomorphic motion. 
In particular for
 $\la_1, \la_2\in \La$, $f_{\la_1}\rest{\mathcal{R}^s_{\la_1}}$ is {\em topologically} conjugate to 
 $f_{\la_2}\rest{\mathcal{R}^s_{\la_2}}$, that is, the induced 
  conjugacy   $\mathcal{R}^s_{\la_1} \cv \mathcal{R}^s_{\la_2}$ is a homeomorphism. 
 
 Since regular points are u- and s- exposed we obtain the following corollary, which contains
 Theorem \ref{theo:pesin strong}.  The conclusion about transversality is not obvious and will be proved afterwards.
 
 \begin{cor}\label{cor:transverse}
 Let $(f_\la)_{\la\in \La}$ 
be a weakly stable substantial family of polynomial automorphisms of $\cd$ of dynamical degree $d\geq 2$. 
Then regular  points move under a strongly unbranched holomorphic motion. Furthermore, transverse regular points remain transverse throughout the family. 
\end{cor}
 
 The following corollary is a first step towards Theorem \ref{theo:hyp}.
 
 \begin{cor}
 If $(f_\la)_{\la\in \La}$ is a weakly stable substantial family of polynomial automorphisms of $\cd$ and if for some $\lo\in \La$, 
 $f_\lo$ is uniformly hyperbolic on $J^*_\lo$, then for every $\la\in \La$, 
 $f_\la\rest{J^*_\la} $ is  topologically conjugate to $f_\lo\rest{J^*_\lo} $.
 \end{cor}
 
 Indeed, just observe that for a hyperbolic map, all points in $J^*$ are regular. 
 
\subsection{Proofs of Theorem \ref{thm:regular} and Corollary \ref{cor:transverse}} 

 \begin{proof}[Proof of Theorem \ref{thm:regular}]
 The plan of the proof is the following:  
 we start by treating the particular case of points belonging to stable manifolds of saddle points.
 Using   the results of \S \ref{subs:param2},  we work locally in $\La$ to show
  that the branched holomorphic 
 motion of $J^*$ is unbranched at 
  $s$-regular and exposed points. 
  Then, using the global area bounds from \S \ref{subs:area} we show that regular points remain 
 regular in the family, which allows to conclude the proof.
 
 \medskip
 
 \noindent{\bf Step 0.} A particular case.
 
 \medskip
 
 Here we prove the following lemma, which is essentially contained in \cite{dl}. 
 
 \begin{lem}\label{lem:particular}
  Let $(f_\la)_{\la\in \La}$  be a weakly stable substantial family of polynomial automorphisms of $\cd$. 
  Let $p: \La\cv\cd$ be such that for every $\la$, $p(\la)\subset K_\la$. 
  Assume that for some $\lo\in \La$, $p(\lo)$ belongs to the stable manifold of a saddle point $m(\lo)$ 
  (which necessarily persists as $m(\la)$ in the family). Then for every $\la\in \La$, $p(\la)\in W^s(m(\la))$. 
  
  If in addition, $p(\lo)$ is exposed inside $W^s(m(\lo))$, then the branched motion of $J^*$ is unbranched along 
  $\widehat p$ and $p(\la)$ remains exposed throughout the family. 
\end{lem}

Of course, the same result holds for unstable manifolds. Recall that $p(\la)$ is exposed inside $W^s(m(\la))$ if and only if 
$p(\la)$ is a limit of homoclinic or heteroclinic intersections for the intrinsic topology of $W^s(m(\la))$.

\begin{proof}
The sequence of iterates $\widehat f^n\lrpar{\widehat p}$ is locally uniformly bounded in $\La\times\cd$. 
Pick a cluster value $\widehat r$ of this sequence. Then $r(\lo) = m(\lo)$ and $\widehat r\subset \widehat K$. 
Then by Theorem \ref{thm:unbranched}, $r\equiv m$, so we conclude that for every $\la$, $p(\la)\in W^s(m(\la))$.

To get the second conclusion, note that for $\la = \lo$, $p(\lo) = \lim t_k(\lo)$ is a limit of homoclinic intersections, 
in the intrinsic topology of $W^s(m(\lo))$. By Theorem \ref{thm:unbranched} $t_k(\lo)$ 
admits a unique continuation $t_k$ to $\La$ as a homoclinic intersection. 
Let $\Delta_\lo\subset W^s(m(\lo))$ be a disk containing $p(\lo)$.
By the persistence of stable manifolds of saddle points, 
there exists a neighborhood $N$ of  $(\lo, p(\lo))$ in $\La\times \cd$ and a smooth surface $\widehat W$ in $N$ such that 
$\widehat W\cap \cd_\lo = \Delta_\lo$ and $\widehat 
W\subset \widehat W^s(\widehat m)$. Now there are two cases: either $p(\lo)$ is itself a homoclinic 
intersection, and we conclude by Theorem \ref{thm:unbranched}. Otherwise  $p(\la)$ is always distinct from $t_k(\la)$, and applying the 
Hurwitz Theorem inside $\widehat W$, we conclude that when $k\cv\infty$, 
$t_k \cv p$ in a neighborhood of $\lo$, hence everywhere by analytic continuation. 
We conclude that $p(\lo)$ admits a unique continuation $p$ staying in $K$

To conclude that the   branched motion of $J^*$ is unbranched along $p$ at all parameters, 
it suffices to show that $p(\la)$ remains exposed inside  $W^s(m(\la))$. For this, it is enough to 
show that $G^-\rest{W^s(m(\la))} \not\equiv  0$ in any neighborhood of $p(\la)$, which  
follows directly  from Lemma \ref{lem:harnack}.  The proof is complete. 
\end{proof}

Let us note for future reference  the following consequence of this lemma. 

\begin{cor}\label{cor:particular}
Let $p(\lo)$ be a s-regular point of size $r$ and   $t(\lo)$ be an 
intersection between 
$W^s_r(p(\lo))$ and $W^u(m(\lo))$, where $m(\lo)$ is a saddle point. Then there exists a unique continuation $t$ of $t(\lo)$ such that 
$\widehat t\subset \widehat K$, and the branched motion of $J^*$ is unbranched along $\widehat t$.
\end{cor}

\begin{proof}
In virtue of Lemma \ref{lem:particular}, it is enough to show that $t(\lo)$ is exposed inside $W^u(m(\lo))$. For this, recall
 that   $W^s_r(p(\lo))$ is the limit of a sequence  $W^s_r(p_n(\lo))$ of local stable manifolds of saddle points. Therefore, 
 by the persistence of proper intersections,  
  $t(\lo)$ is the  limit  in the intrinsic topology of $W^u(m(\lo))$ of a sequence of 
  heteroclinic intersections with $W^s_r(p_n(\lo))$, 
  and we are done.
  \end{proof}

 \medskip
\noindent{\bf Step 1.} The branched motion is unbranched at s-regular and exposed points. 
  
 \smallskip
 
Let $\lo\in\widetilde\La\Subset \La$ and $p(\lo)$ be s-regular and exposed for $f_\lo$. We want to 
construct a natural continuation  of $p(\lo)$. 
 Let  $r_0>0$ be such that there exists a sequence of distinct saddle points $p_n\cv p$ with local stable  manifolds of size $r'_0 := 2r_0$. 
  Extracting a subsequence we assume that $(\widehat p_n)$ converges to some $\widehat p$ in $\La\times \cd$ 
   (later on we will see that this limit is unique).
	It follows from Proposition \ref{prop:surface} that for $\abs{\la-\lo}<\delta=\delta(r_0, \widetilde\La)$, 
	$W^s_\loc( p_n(\la))$ is of size $r_0$, 
 therefore $p(\la)$ is s-regular\footnote{We see that it is crucial  here  that $\delta(r_1, r_2)$ in Proposition \ref{prop:surface} depends only on $r_1$ and $r_2$.}.  
 
Our goal here  is to show that   if $q:\La\cv \cd$ is such that $q(\lo)=p(\lo)$ and $q(\la)\in K_\la$ for every $\la$, then $q(\la)=p(\la)$ for every $\la$.

\begin{claim}
There exists a  neighborhood $N = \tub\lrpar{\widehat p , r_0}\cap (D(\lo, \delta(r_0))\times \cd)
$ of $(\lo, p(\lo))$ in $\La\times \cd$ and a smooth hypersurface 
 $\widehat{W}^s_{r_0}\lrpar{\widehat{p}}$ in $N$ such that the sequence of hypersurfaces 
 $\widehat W^s_{r_0}\lrpar{\widehat{p}_n}$ given by Proposition \ref{prop:surface}
 converges to  $\widehat{W}^s_{r_0}\lrpar{\widehat{p}}$ with multiplicity 1 in $N$.
\end{claim}

\begin{proof}
Indeed  the volumes of  $\widehat W^s_{r_0}\lrpar{\widehat{p}_n}$  are uniformly bounded in $N$, so 
 we may extract converging subsequences by Bishop's Theorem \ref{thm:bishop}. 
 Fix such a subsequence $\widehat W^s_{r_0}\lrpar{\widehat{p}_{n_j}}\cv W$. 
  Since the unstable directions $E^u(p_n(\lo))$ converge, for $\abs{\la-\lo}<\delta$,  
 $W^s_{r_0}(p_n(\la))$ is a graph of slope at most 2 over a fixed direction for large $n$. 
 In particular the convergence is of multiplicity 1. 
  By Lemma \ref{lem:smooth}, $W$ is smooth, and from Proposition \ref{prop:cv} we get that the sequence 
  $\widehat W^s_{r_0}\lrpar{\widehat{p}_n}$ actually converges. Notice that by construction, for 
 $\abs{\la-\lo}<\delta$, we have that $W^s_{r_0}(p(\la)) = \widehat{W}^s_{r_0}\lrpar{\widehat{p}} \cap \cc^2_\la$. 
\end{proof} 

As a consequence of this claim and the Hurwitz Theorem we get  the following:

\begin{claim}\label{claim:2}
There exists $\eta=\eta(r_0, \widetilde \La)>0$ such that if $q$ is a holomorphic map  $\La\cv\cc^2$  such that  $\widehat q\subset \widehat K$ and 
  $q(\lo)\in W^s_{r_0/2}(p(\lo))$, 
then for $\abs{\la-\lo}<\eta$, $q(\la)\in W^s_{r_0}(p(\la))$.
\end{claim}

\begin{proof} 
Discarding at most one value of $n$ if needed, we may assume that $\widehat q$ is disjoint from 
  $\widehat W^s_{r_0}\lrpar{\widehat{p}_n}$.
Since $q(\lo) \in \widehat{W}^s_{r_0/2}\lrpar{\widehat{p}}$, 
by the Cauchy estimates,   there exists $\eta=\eta(r_0)>0$ so that if $\abs{\la-\lo}<\eta$, then the point $q(\la)$ stays in  $B(p(\la), r_0)$. 

Now we have that $\widehat q \subset \widehat{W}^s_{r_0}\lrpar{\widehat{p}}$. Indeed otherwise these two manifolds would have a proper
intersection at $(\lo, q(\lo))$, and by persistence  of proper intersections we would get that 
  $\widehat q$ intersects $\widehat W^s_{r_0}\lrpar{\widehat{p}_n}$, a contradiction. 
\end{proof}

  \medskip
  
If $p(\lo)$ is a transverse regular point, this is enough to conclude. Indeed, applying the same reasoning 
in the unstable direction we get that $ q(\la) \subset {W}^u_{r_0}\lrpar{{p(\la)}}$ for $\la$ close to $\lo$. Now the  intersection 
  $\widehat{W}^s_{r_0}\lrpar{\widehat{p}}\cap \widehat{W}^u_{r_0}\lrpar{\widehat{p}}$ is transverse near $(\lo, p(\lo))$, 
  therefore it coincides with
   $\widehat p$. We conclude that $\widehat p  = \widehat q$ near $\lo$, hence everywhere, which was the desired result. 
  
  \medskip
  
Let us now deal with the general case.
 
\begin{claim}
Let  $p(\lo)$ be  s-regular and exposed.
If $q:\La\cv \cd$ is such that $q(\lo)=p(\lo)$ and $q(\la)\in K_\la$ for every $\la$, then $q(\la)=p(\la)$ for every $\la$.
\end{claim}

\begin{proof}
By Definition-Proposition \ref{defi:exposed}, for a given  saddle point $m(\lo)$, there exist a sequence of 
transverse intersection points $(t_k(\lo))$ between $W^s_{r_0/2}(p(\lo))$ and $W^u(m(\lo))$ such that $t_k(\lo)\cv p(\lo) = q(\lo)$.
By Corollary \ref{cor:particular}, there exists a unique holomorphic continuation $t_k$ of $t_k(\lo)$ 
with the property  that for every $\la\in \La$, $t_k(\la)\in W^u(m(\la))\cap K(\la)$. 
In particular if $p(\lo)$ itself belongs to $W^u(m(\lo))$ we are done, so let us assume that $p(\lo)\notin W^u(m(\lo))$.
By the previous Claim \ref{claim:2}, for every $\la$ so that $\abs{\la -\lo}< \eta(r_0)$, the point $t_k(\la)$  belongs to
 $W^u(m(\la))\cap W^s_{r_0}(p(\la))$ and $q(\la)$ belongs to $W^s_{r_0}(p(\la))$.
 To conclude, we observe that since $p(\lo)=q(\lo)\notin W^u(m(\lo))$, applying Corollary \ref{cor:particular} again, we deduce that 
  $t_k(\la)$ is disjoint from $q(\la)$ for every $\la\in\La$.
  Working inside    $\widehat{W}^s_{r_0}\lrpar{\widehat{p}}$, which is a smooth complex surface, 
we can apply the Hurwitz Theorem to conclude that the sequence  $\widehat t_k$ converges to $\widehat q$.
 Therefore   the continuation $p$ is unique, which was the result to be proved.
\end{proof}

 \noindent {\bf Step 2.} The branched motion preserves s-regularity and exposure.
 
 \smallskip
 
Fix a relatively compact open set $\widetilde \La\Subset \La$. 
For $\lo\in \widetilde \La$, let $p(\lo)\in J^*_\lo$ be  s-regular and exposed. Thus there exists a sequence of saddle points $(p_n)$ 
with stable manifolds of size $r_0$, such that  $p_n(\lo) \cv p(\lo)$.  By Proposition \ref{prop:cvsize},
 $W^s_{r_0}(p_n(\lo))$ converges to $W^s_{r_0}(p(\lo))$ with multiplicity 1, that is, $W^s_{r_0}(p_n(\lo))$
  is a graph over $W^s_{r_0}(p(\lo))$ for large $n$. Since $p(\lo)$ is 
s-exposed,   for every $r$ we get that   $G^-\rest{W^s_r(p(\lo))}$ is not identically 0.   

 Let $\psi^s_{\lo, n}$ be a stable parameterization of $W^s (p_n(\lo))$ with $\|{(\psi^s_{\lo, n})'(0)} \|=1$, and 
 $(\psi_{\la, n}^s)$ be the natural continuation of $\psi_{\la_0, n}^s$, as defined in \S\ref{subs:param1}. 
 Recall the notation $h_\la^s$ for the holomorphic motion inside stable manifolds, 
 viewed inside  the parameterizations $(\psi_{\la, n}^s)$. Recall that $h_\la^s$ satisfies $h_\la^s(0) = 0$ and $h^s_\la(1)=1$ 
 (hence also the estimate \eqref{eq:holmotion}). 
By \eqref{eq:holmotion},  there exists $c>0$ such that for every $\la\in \widetilde\La$,
$(h_\la^s)^{-1}(D(0, cr_0))\subset D(0, r_0/4)$. 
 Without loss of generality, we can assume that $c<1/8$. 
Choose $c'<c$ so small that   for every $\la\in \widetilde \La$, $h_\la^s(D(0, c'r_0))\Subset D(0, cr_0)$

Fix a pair of holomorphic disks $D_1\Subset D_2$
 in $W^s_{c'r_0/4}(p(\lo))$, with $p(\lo)\in D_1$.
 Set $m =\modul(D_2\setminus \overline D_1)$, 
 and for $i=1,2$, put  $g_i = \sup G_\lo^-\rest{D_i}$. By the continuity of $G^-_\lo$ 
  and the multiplicity 1 convergence, for large $n$ 
  we can lift $D_1$ and $D_2$ to holomorphic disks $D_{1, n}$ and $D_{2, n}$ in $W^s_{c'r_0/4}(p_n(\lo))$, such that  
 $\modul(D_{2, n}\setminus \overline D_{1, n})\cv m$ and $\sup G_\lo^-\rest{D_{i, n}}\cv g_i$.
 From Lemma \ref{lem:unstable}
 we infer that $\psi^s_{\lo, n} (D(0, c'r_0)) \supset W^s_{c'r_0/4}(p_n(\lo))$, so in particular  
 $\psi^s_{\lo, n} (D(0, c'r_0))$ contains $D_{1, n}$ and $D_{2, n}$. 
 
 \medskip
 
Now  fix another parameter $\la_1\in \widetilde \La$. 
 By the first step of the proof we know that $p_n(\la_1)\cv p(\la_1)$. 
 Applying Proposition \ref{prop:uniform} we infer that there exist positive constants  $r_1$, $g$ and $A$  such
that $W^s_{r_1}(p_n(\la_1))$ is a submanifold properly embedded into $B(p_n(\la_1), r_1)$, contained in 
$h_{\la_1}(D_{1,n})$,  with area at most $A$, and 
$\sup \big({G^-_{\la_1}\rest{W^s_{r_1}(p_n(\la_1))}}\big)\geq g$. 
Using Bishop's Theorem, we extract a subsequence $n_j$ so that $(W^s_{r_1}(p_{n_j}(\la_1)))_j$ converges to some analytic set 
$W\ni p(\la_1)$ with $\sup \big({G^-_{\la_1}\rest{W}}\big)\geq g$. 

\medskip

The main step of the proof is the following lemma. 

\begin{lem}\label{lem:multiplicity}
The multiplicity of convergence of $W^s_{r_1}(p_{n_j}(\la_1))$ to $W$  is equal to 1. 
\end{lem}

Before establishing the lemma let us show how to conclude the proof of Step 2. Recall that by 
 the Maximum Principle  $W^s_{r_1}(p_{n_j}(\la_1))$ is a holomorphic disk. Since the multiplicity of convergence is 1,
  we deduce  from 
Proposition \ref{prop:disk} that $W$ is smooth.  Also  Proposition \ref{prop:cv} 
  implies that the sequence $W^s_{r_1}(p_n(\la_1))$ actually converges. 

Since $W$ is smooth at $p(\la_1)$ and the multiplicity of convergence is 1, we see that in any small neighborhood of $p(\la_1)$,  
$W^s_{r_1}(p_n(\la_1))$ is a graph over $W$ for large $n$. In particular $W^s_{r_1}(p_n(\la_1))$ has size uniformly bounded 
from below, and 
therefore  $p(\la_1)$ is regular. We   already observed that $\sup \big({ G^-_{\la_1}\rest{W}}\big)\geq g>0$, and the same holds in 
any neighborhood of $p(\la_1)$ by choosing  a smaller $r_0$  at the beginning.   Hence $p(\la_1)$ 
is s-exposed, which finishes the proof of Step 2.

\begin{proof}[Proof of Lemma \ref{lem:multiplicity}]
For notational ease we put $n_j=n$. Let $k$ be the multiplicity of convergence of $W^s_{r_1}(p_{n}(\la_1))$ to $W$.
 By Proposition \ref{prop:uniformparam}, we know that for every $\la\in \widetilde \La$, 
  $\|{\psi_{\la, n}^s }\| \leq M$ on $\widetilde \La \times D(0, cr_0)$, so we can extract a converging subsequence (still denoted by $n$)
 to a limiting map $\varphi_\la(\cdot)$. Notice that for $\la = \lo$, 
 $(\psi_{\la, n}^s)_n $ converges on $D(0, r_0/4)$ to an injective map $D(0, r_0/4)\cv W^s_r(p(\lo))$
    
We recall that  $D_{2,n}\subset  \psi^s_{\lo, n}(D(0, c'r_0))$, so by   definition of $c'$ we get that  for every $\la$, 
 $\psi_{\la, n}^s\lrpar{D(0, cr_0)}$ contains $h_{\la}(D_{2,n})$.  It follows that for $\la = \la_1$,    
 $W^s_{r_1}(p_n(\la_1))\subset \psi_{\la_1, n}^s\lrpar{D(0, cr_0)}$. 
 
 Furthermore, since $\modul(D_{2, n}\setminus D_{1,n}) \cv m>0$, there exists   a uniform $c''<c$ such that 
$W^s_{r_1}(p_n(\la_1))\subset \psi_{\la_1, n}^u\lrpar{D(0, c''r_0)}$. It follows that $\varphi_{\la_1}$ is non-constant and that 
  the component $\om$ of $0$ in $\varphi_{\la_1}^{-1}(W)$, is such that  $\varphi_{\la_1}:\om\cv W$ is proper. 
Its degree is equal to the mutiplicity of convergence $k$.  

\medskip

Since $G^+_{\la_1}\rest{W}$ (resp. $G^+_{\la_1}\circ \varphi_{\la_1}$)
 is  continuous and not harmonic, its Laplacian is nonzero and gives no mass to points.
  By the s-exposure assumption,  Definition-Proposition \ref{defi:exposed} ensures the existence of 
  a saddle point $m(\la_1)$ whose unstable manifold intersects transversally $W$ at a certain point $q(\la_1)$ which is a 
  regular value of $\varphi_{\la_1}$. 
 
Now if $k>1$, there exist two distinct points $a$ and  $b$ in $\om \subset D(0, cr_0)$ 
such that $\varphi_{\la_1} (a) = \varphi_{\la_1}(b) = q(\la_1)$. 
Thus there exists $a_n \cv a$ (resp. $b_n\cv b$) such that $\psi_{\la_1, n}^s(a_n)$ (resp. $\psi_{\la_1, n}^s(b_n)$ ) are intersection points of 
$W^u(m(\la_1))$ and $W^s_{r_1}(p_n(\la_1))$  converging to $q(\la_1)$. 

\medskip

To conclude the proof, we flow back to $\lo$ using the holomorphic motion to obtain a contradiction with Corollary \ref{cor:particular}. 
The details are as follows. 
Consider the continuations of the heteroclinic intersections  $\psi_{\la_1, n}^s(a_n)$
and $\psi_{\la_1, n}^s(b_n)$ for $\la\in \widetilde \La$. 
Notice that they stay in a compact piece of $W^u(m(\la))$. For $\la=\lo$, the corresponding points are 
$\psi_{\lo, n}^s((h_{\la_1}^s)^{-1}(a_n))$ and $\psi_{\lo, n}^s((h_{\la_1}^s)^{-1}(b_n))$, which converge respectively 
to $\varphi_\lo((h_{\la_1}^s)^{-1}(a))$ and $\varphi_\lo((h_{\la_1}^s)^{-1}(b))$. Now
 $(h_{\la_1}^s)^{-1}(a)$ and $(h_{\la_1}^s)^{-1}(b)$ are distinct and by definition of $c$, they belong 
to $D(0, r_0/4)$. On this disk, $\varphi_\lo$ is injective therefore  $\varphi_\lo((h_{\la_1}^s)^{-1}(a))$ and $
\varphi_\lo((h_{\la_1}^s)^{-1}(b))$ are distinct 
 intersection points between $W^s_r(p(\lo))$ and $W^u(m(\lo))$ with continuations colliding at 
$\la_1$. This contradicts Corollary \ref{cor:particular}, and concludes the proof of the lemma. 
\end{proof}

\begin{rmk}
The  proof does not give any estimate on the size $r_1$ of the stable manifold $W^s(p(\la_1))$  at $p(\la_1)$. 
 Indeed, $r_1$ depends on the size of $W$ at $p(\la_1)$, upon which we have no control 
 (the only information we have is    a  local area bound).
In particular it is unclear whether $r_1$ depends only on $r_0$ and $\widetilde \La$. 
 \end{rmk}

 \begin{rmk}
 As opposed to Step 1 of the proof, 
 Step 2 does  not become significantly easier if we assume that $p(\lo)$ is regular instead of s-regular and exposed. 
 Indeed, the  whole point is to prove that $p(\la)$ remains s-regular throughout $\La$.
 \end{rmk}

\noindent{\bf Step 3.} Conclusion.
  \smallskip
 
 We have shown in Steps 1 and 2 that if $p(\lo)$ is s-regular and exposed, then it admits a unique holomorphic continuation 
 $\widehat p\subset \widehat K$ such that for every $\la$, $p(\la)$ is s-regular and exposed, too. Thus the branched motion of $J^*$ must be unbranched, in particular continuous, at $(\la, p(\la))$. In particular 
  $p(\la)$ cannot collide with the continuation of any other point in $J^*$, so we indeed have a continuous holomorphic motion of $\mathcal{R}^s$. This finishes the  proof of Theorem \ref{thm:regular}.  
  \end{proof}
  
  Let us note for future reference the following consequence of the proof.
  
 \begin{prop} \label{prop:strong regular}
  Let $(f_\la)$ be a weakly stable holomorphic family of polynomial automorphisms. Assume that for $\la = \lo$, 
   $p(\la_0)$ is a regular point and $(p_n(\lo))$ is a sequence  of saddle points converging to $p(\lo)$ such that 
    $W^s(p_n(\lo))$ (resp. $W^u(p_n(\lo))$) is of size $r_0$ at $p_n(\lo)$. 
    Then for every $\la_1\in \La$, there exists $r_1>0$ such that 
     $p_n(\la_1)\cv p(\la_1)$ and $W^s(p_n(\la_1))$  (resp. $W^u(p_n(\la_1))$)
      is of size $r_1$ at $p_n(\la_1)$. 
 \end{prop}

We now show that regular points (resp. transverse regular points) remain regular (resp. transverse).

\begin{proof}[Proof of Corollary \ref{cor:transverse}] 
It follows from Theorem \ref{thm:regular} that the regular points move without collision, 
and remain s- and u-regular and exposed in both directions.  Let $p(\lo)$ be regular relative to  $f_\lo$. Then  for every $\la\in \La$, $p(\la)$ is s- and u- regular, so it possesses local stable and unstable manifolds. If $W^u_\loc (p(\la_1))$ was to coincide with $W^s_\loc (p(\la_1))$, we would get that 
 $W^u_\loc (p(\la_1)) = W^s_\loc (p(\la_1)) \subset K_{\la_1}$, thus contradicting s- and u-exposure. Therefore $p(\la_1)$ is regular. 
 
 To show that transverse regular points stay transverse, recall from the proof of Theorem \ref{thm:regular} that if $p(\lo)$ is regular, then 
 $W^s_\loc(p(\lo))$ and   $W^u_\loc(p(\lo))$ can be locally continued as smooth surfaces $\widehat W^s_\loc(\widehat p)$ and 
 $\widehat W^u_\loc(\widehat p)$ in  
 $N(\lo)\times B(p(\lo),r)$. Now assume that $W^s_\loc(p(\lo))$ and   $W^u_\loc(p(\lo))$ are tangent
  at $p(\lo)$, that is, their intersection multiplicity at $  p(\lo))$ is $m>1$. 
  If this tangency does not persist for nearby parameters, by the persistence of  proper intersections, we get that for nearby $\la$, 
  $W^s_\loc(p(\la))$ and   $W^u_\loc(p(\la))$ intersect at $m$ points counting multiplicities, not all identical to $p(\la)$. 
 
 \medskip 
 
 Consider the intersection $\widehat C = \widehat W^s_\loc(\widehat p)\cap \widehat W^u_\loc(\widehat p)$. This is a curve 
   in $N(\lo)\times B(p(\lo),r)$ such that $\widehat C\cap \cc^2_\lo = \set{p(\lo)}$. 
   One irreducible  component of $\widehat C$ is given by the continuation $\widehat p$, and by assumption there exists another irreducible component $\widehat C'$ of $\widehat C$.
    
   Assume first that $\widehat C'$ 
 	is a  graph $\widehat q$ over $N(\lo)$.  Then, since for every $\la\in N(\lo)$, 
    $  q(\la)\in W^s_\loc(p(\la))\cap W^u_\loc(p(\la)) \subset  K_\la$, we get a collision between $ p$ and a
     holomomorphically moving point $q$ staying in $K$, which contradicts Theorem \ref{thm:regular}.
 
We will reduce the general case to this one by  a classical trick: replacing $\La$ by  a well-suited  
 branched cover   $M\cv \La$. We detail the argument for the   convenience of the reader.
Consider a local  irreducible component of $\widehat C'$ at $(\lo, p(\lo))$, still denoted by  $\widehat C'$ for simplicity. 
Denote by $\varpi: \dd\cv\widehat C'$ a normalization of $\widehat C'$ such that $\varpi(0)  =  (\lo, p(\lo))$. 
Denote respectively  by $\pi_\La$ and $\pi_\cd$ the projection onto the first and second
factors in $\La \times \cd$ and put $M = \dd$ and 
$\la(\mu)  = \pi_\La\circ \varpi (\mu)$. Then we can consider the holomorphic family of polynomial 
automorphisms defined by $(\widetilde f_\mu) := (f_{\la(\mu)})$, which is weakly stable
 (of course,  nothing has changed from the dynamical point of view). 
For $\mu= 0$, the point $p(\la(0))$ is regular and can be continued as a regular point $\mu \mapsto p(\la(\mu))$ as before. 
But now we have another holomorphic continuation of $(0,p(\la(0)))$  in $\widehat K\subset M\times \cd$
 given  by $\mu\mapsto (\mu, \pi_\cd\circ \varpi(\mu))$.
  Thus we arrive at a contradiction and the proof is complete. 
\end{proof}
 
 \begin{rmk}
 A similar argument shows that more generally the order of tangency between local stable and unstable manifolds of regular points is preserved in weakly stable families.  
 \end{rmk}

\section{Propagation of hyperbolicity}\label{sec:hyperbolic}
 
In this section we establish  Theorem \ref{theo:hyp}, that is, we prove that uniform 
hyperbolicity on $J^*$ is preserved in weakly stable families. Let us start with a variation on Definition \ref{defi:usregular}.

\begin{defi}\label{defi:uniformly regular}
We say that $p\in J^*$ is uniformly u-regular  (resp. uniformly s-regular) if there exists $r>0$ such that for every sequence of 
saddle points $(p_n)$ converging to  $p$, $W^u(p_n)$ (resp. $W^s(p_n)$) is of size $r$ at $p_n$. 

Likewise, $p$ is uniformly (resp. transverse) regular
 if it is (resp. transverse) regular and uniformly regular in both stable and unstable directions. 
\end{defi}

 If necessary we will specify the size appearing in the definition by saying that ``$p\in J^*$ is uniformly u-regular of size $r$". Recall from Proposition \ref{prop:cvsize} that if $p$ is uniformly u-regular of size $r$, then it has a well defined local unstable manifold $W^u_r(p)$ and 
 that if $p_n\cv p$ is any sequence of saddle points, $W^u_r(p_n)$ converges to $W^u_r(p)$ with multiplicity 1.  

\medskip

If $f$ is uniformly hyperbolic on $J^*$, then every $p\in J^*$ is uniformly regular and transverse. Interestingly enough, the converse is true. 

\begin{prop}\label{prop:criterion}
Let $f$ be a polynomial automorphism of $\cd$ with   dynamical degree $d\geq 2$. If every point in $J^*$ is   uniformly regular and transverse then $f$ is uniformly hyperbolic on $J^*$.
\end{prop}

The main step of the proof is the following  lemma. 

\begin{lem}\label{lem:lamination}
Let $f$ be a polynomial automorphism of $\cd$, such that every point in $J^*$ is uniformly u-regular. 
Then there exists a neighborhood $N$ of $J^*$ such that the restriction to $N$ of $\bigcup_{p\in J^*} W^u_\loc(p)$ forms  a lamination. 
\end{lem} 

\begin{proof}
Let us start by showing that the size of unstable manifolds is uniformly bounded from below. 
For this, notice that Definition \ref{defi:uniformly regular} may be reformulated as follows: $p$ is uniformly u-regular if there 
exists $r>0$ and $\e>0$ such that if $q\in B(p,\e)$ is any saddle point, then $W^u(q)$ is of size $r$ at $q$. Then
 by compactness of $J^*$, we can cover $J^*$ with finitely many such balls, and  deduce that if every point in $J^*$ is uniformly u-regular, then the size of unstable manifolds of saddle points 
 is uniformly bounded form below, as claimed.
 
From this point, the remainder of the proof is classical. As observed above, for every $p\in J^*$ there 
exists $r>0$ and $\e>0$ such that if $q\in B(p,\e)$ is any saddle point, then $W^u_r(q)$ is of size $r$ at $q$. Taking $\e$ smaller if needed, 
we may assume that $W^u_r(q)$ is closed in $B(p, \e)$. Furthermore, any two such local unstable manifolds of saddle points are disjoint or 
coincide. Thus taking the closure, we get that $\overline{\bigcup  W^u_r(q)}\cap B(p, \e)$ is a lamination in $B(p,\e)$, 
where the union ranges over all saddle points lying  in $B(p, \e)$. The result follows.
\end{proof}

\begin{proof}[Proof of Proposition \ref{prop:criterion}]
The result is  essentially a  direct consequence of Theorem 8.3. in \cite{bs8}, 
which asserts that if there exist    laminations  
of $J^+$ and $J^-$ in a neighborhood of $J^*$, 
which are transverse at every point of $J^*$ then $f$ is uniformly hyperbolic on $J^*$. 

In our situation, the existence of stable and unstable   laminations $\mathcal L^s$ and $\mathcal L^u$ 
is guaranteed by Lemma \ref{lem:lamination}, while these laminations are transverse at every point of $J^*$ by assumption. 

Unfortunately, this is slightly different  from the hypotheses of \cite[Thm. 8.3]{bs8} 
because we do not know that the lamination $\mathcal L^u$ 
  fills up the whole $J^-$ in a neighborhood of $J^*$. However, 
the reader will easily check that the only place in  the  proof of \cite{bs8} where this assumption is used 
 is to ensure that for every $p\in J^*$, $W^u_\loc(p)$ is contained
 in a leaf of $\mathcal L^u$, which is trivially satisfied in our case. Hence the result
 applies and we are done.
 \end{proof}

We now have all the necessary ingredients for Theorem \ref{theo:hyp}.

\begin{proof}[Proof of Theorem \ref{theo:hyp}]
By assumption, every point in $J^*_\lo$ is uniformly regular and transverse. 
From Corollary \ref{cor:transverse}  we deduce that  $J^*$ moves holomorphically and all points remain regular and transverse.
Proposition \ref{prop:strong regular} implies  that strong s- and u-regularity are preserved as well. Therefore, for every $\la\in \La$, 
every point in $J^*_\la$ is uniformly regular and transverse, so the result follows from Proposition \ref{prop:criterion}. 
\end{proof}

The concept of quasi-expansion, developed in \cite{bs8} has been a source of inspiration for the techniques in this paper. A polynomial automorphism of $\cd$ of dynamical degree $d\geq 2$ is {\em quasi-expanding} if there exists positive constants $r$ and  
 $A$   such that for every saddle point $p$, $W^u_r(p)$ is properly embedded in $B(p, r)$, of 
 area at most $A$ and  for every $\delta>0$ there exists $\eta>0$ such that 
   $\sup\big({G^+\rest{W^u_\delta(p)}\big)}\geq \eta$ (see \cite[Cor 3.5]{bs8} for this definition). There is a parallel notion of 
    {\em quasi-contraction} in the stable direction.
 
 \medskip
 
It is worthwhile to state the  following result   of independent interest. 
 
 \begin{prop}\label{prop:qe}
 Let $(f_\la)_{\la\in \La}$ be a weakly stable and substantial holomorphic family of polynomial automorphisms. If there exists $\lo\in \La$ such that $f_\lo$ is quasi-expanding, then $f_\la$ is quasi-expanding for every $\la\in \La$.  
 \end{prop}
 
 \begin{proof}
 For $\la=\lo$, let    $r$,  $A$  be the uniform constants provided  by the definition of quasi-expansion. Let  
 $p(\lo)$ be a saddle point.  By \cite[Thm 3.1]{bs8}, the modulus of the annulus  $W^u_r(p(\lo))\setminus W^u_{r/2}(p(\lo))$ is bounded from below by a constant $m$ depending only on $A$ and $r$. By the Hölder continuity property of $G^+$ we get that 
 $\sup\big({G^+\rest{W^u_r(p)}}\big)\leq g_2(r)$ and by the definition of quasi-expansion, 
 $\sup\big({G^+\rest{W^u_{r/2}(p)}\big)}\geq g_1>0$. Therefore applying  Proposition \ref{prop:uniform} we obtain for every $\la\in \La$ positive  constants $r'$, $A'$ and $g'$ such that for every saddle point $p'$ for $f_{\la'}$   , $W^u_{r'}(p')$ is properly embedded in $B(p', r')$, of 
 area at most $A'$ and    $\sup\big({G^+\rest{W^u_{r'}(p')}\big)}\geq g'$. Finally, Theorem 3.4 in \cite{bs8} implies that $f_\la$ is 
 quasi-expanding. 
  \end{proof}

\end{document}